\pdfoutput=1
\documentclass[10pt,oneside]{amsart}
\usepackage[
paper=a4paper,
headsep=20pt,text={136mm,208mm},centering,includehead
]{geometry}
\usepackage{graphicx}
\usepackage[dvipsnames,usenames]{xcolor}
\usepackage[colorlinks=true,urlcolor=ForestGreen,linkcolor=ForestGreen,citecolor=ForestGreen]{hyperref}
\hypersetup{
	linkbordercolor={1 0 0}, % set to white
	citebordercolor={0 1 0} % set to white
}
\usepackage{wasysym}
\usepackage{amssymb,mathtools,mathrsfs,amsmath,amsthm,stmaryrd}
\usepackage{bbm}
\usepackage[UKenglish]{babel}
\usepackage[noabbrev]{cleveref}
\usepackage{tikz}
\usetikzlibrary{trees,decorations.pathmorphing,decorations.markings,matrix,shapes}
\usepackage[backgroundcolor=green,linecolor=green,bordercolor=green]{todonotes}
\usepackage{enumerate}

\iftrue
\makeatletter
\def\@settitle{%
	\vspace*{-20pt}
	\begin{flushleft}%
		\baselineskip14\p@\relax
		\normalfont\bfseries\LARGE
		\@title
	\end{flushleft}%
}
\def\@setauthors{%
	\begingroup
	\def\thanks{\protect\thanks@warning}%
	\trivlist
	\large \@topsep30\p@\relax
	\advance\@topsep by -\baselineskip
	\item\relax
	\author@andify\authors
	\def\\{\protect\linebreak}%
	\authors
	\ifx\@empty\contribs
	\else
	,\penalty-3 \space \@setcontribs
	\@closetoccontribs
	\fi
	\normalfont
	\@setaddresses
	\endtrivlist
	\endgroup
}
\def\@setaddresses{\par
	\nobreak \begingroup\raggedright
	\small
	\def\author##1{\nobreak\addvspace\smallskipamount}%
	\def\\{\unskip, \ignorespaces}%
	\interlinepenalty\@M
	\def\address##1##2{\begingroup
		\par\addvspace\bigskipamount\noindent
		\@ifnotempty{##1}{(\ignorespaces##1\unskip) }%
		{\ignorespaces##2}\par\endgroup}%
	\def\curraddr##1##2{\begingroup
		\@ifnotempty{##2}{\nobreak\noindent\curraddrname
			\@ifnotempty{##1}{, \ignorespaces##1\unskip}\/:\space
			##2\par}\endgroup}%
	\def\email##1##2{\begingroup
		\@ifnotempty{##2}{\smallskip\nobreak\noindent E-mail address%
			\@ifnotempty{##1}{, \ignorespaces##1\unskip}\/:\space
			\ttfamily##2\par}\endgroup}%
	\def\urladdr##1##2{\begingroup
		\def~{\char`\~}%
		\@ifnotempty{##2}{\nobreak\noindent\urladdrname
			\@ifnotempty{##1}{, \ignorespaces##1\unskip}\/:\space
			\ttfamily##2\par}\endgroup}%
	\addresses
	\endgroup
	\global\let\addresses=\@empty
}
\def\@setabstracta{%
	\ifvoid\abstractbox
	\else
	\skip@25\p@ \advance\skip@-\lastskip
	\advance\skip@-\baselineskip \vskip\skip@
	\box\abstractbox
	\prevdepth\z@ % because \abstractbox is a vtop
	\vskip-15pt
	\fi
}
\renewenvironment{abstract}{%
	\ifx\maketitle\relax
	\ClassWarning{\@classname}{Abstract should precede
		\protect\maketitle\space in AMS document classes; reported}%
	\fi
	\global\setbox\abstractbox=\vtop \bgroup
	\normalfont\small
	\list{}{\labelwidth\z@
		\leftmargin0pc \rightmargin\leftmargin
		\listparindent\normalparindent \itemindent\z@
		\parsep\z@ \@plus\p@
		
	}%
	\item[\hskip\labelsep\bfseries\abstractname.]%
}{%
	\endlist\egroup
	\ifx\@setabstract\relax \@setabstracta \fi
}

\def\ps@headings{\ps@empty
	\def\@evenhead{%
		\setTrue{runhead}%
		\normalfont\scriptsize
		\rlap{\thepage}\hfill
		\def\thanks{\protect\thanks@warning}%
		\leftmark{}{}}%
	\def\@oddhead{%
		\setTrue{runhead}%
		\normalfont\scriptsize
		\def\thanks{\protect\thanks@warning}%
		\rightmark{}{}\hfill \llap{\thepage}}%
	\let\@mkboth\markboth
}\ps@headings

\def\section{\@startsection{section}{1}%
	\z@{-1.2\linespacing\@plus-.5\linespacing}{.8\linespacing}%
	{\normalfont\bfseries\Large}}
\def\subsection{\@startsection{subsection}{2}%
	\z@{-.8\linespacing\@plus-.3\linespacing}{.3\linespacing\@plus.2\linespacing}%
	{\normalfont\bfseries\large}}
\def\subsubsection{\@startsection{subsubsection}{3}%
	\z@{.7\linespacing\@plus.1\linespacing}{-1.5ex}%
	{\normalfont\bfseries}}
\def\@secnumfont{\bfseries}
\makeatother
\fi %\iftrue/false for amsart.cls hack

\makeatother

\newtheorem{theorem}{Theorem}[section]
\newtheorem*{theorem*}{Theorem}

\newtheorem{proposition}[theorem]{Proposition}
\newtheorem{corollary}[theorem]{Corollary}

\theoremstyle{definition}
\newtheorem{definition}[theorem]{Definition}

\numberwithin{equation}{section}

\newcommand{\ckh}{{CKh}}
\newcommand{\kh}{{Kh}}
\newcommand{\cdkh}{CDKh}
\newcommand{\dkh}{DKh}
\newcommand{\vckh}{vCKh}
\newcommand{\vkh}{vKh}
\newcommand{\Q}{\mathbb{Q}}
\newcommand{\Z}{\mathbb{Z}}

\newcommand{\vup}{v^{\text{u}}_+}
\newcommand{\vum}{v^{\text{u}}_-}
\newcommand{\vlp}{v^{\text{l}}_+}
\newcommand{\vlm}{v^{\text{l}}_-}
\newcommand{\vulp}{v^{\text{u/l}}_+}
\newcommand{\vulm}{v^{\text{u/l}}_-}

\newcommand{\dpl}{\overset{\bullet}{+}}
\newcommand{\dm}{\overset{\bullet}{-}}
\newcommand{\gH}{\mathfrak H}

\DeclareRobustCommand{\CloseDef}{%
	\leavevmode\unskip\penalty9999 \hbox{}\nobreak\hfill
	\quad\hbox{$\lozenge$}%
}

\begin{document}
	
\vspace*{-50pt}
\title[Additional gradings and invariants of surfaces]{Additional gradings on generalisations of Khovanov homology and invariants of embedded surfaces}

\author{Vassily Olegovich Manturov}
\address{
Bauman Moscow State Technical University and Chelyabinsk State University, Russian Federation
}
\email{\href{mailto:vomanturov@yandex.ru}{vomanturov@yandex.ru}}

\author{William Rushworth}
\address{
	Durham University, United Kingdom
}
\email{\href{mailto:william.rushworth@durham.ac.uk}{william.rushworth@durham.ac.uk}}

\def\subjclassname{\textup{2010} Mathematics Subject Classification}
\expandafter\let\csname subjclassname@1991\endcsname=\subjclassname
\expandafter\let\csname subjclassname@2000\endcsname=\subjclassname


\begin{abstract}
	We define additional gradings on two generalisations of Khovanov homology (one due to the first author, the other due to the second), and use them to define invariants of various kinds of embeddings. These include invariants of links in thickened surfaces and of surfaces embedded in thickened \(3\)-manifolds. In particular, the invariants of embedded surfaces are expressed in terms of certain diagrams related to the thickened \(3\)-manifold, so that we refer to them as \emph{picture-valued invariants}. This paper contains the first instance of such invariants for \(2\)-dimensional objects.
	
	The additional gradings are defined using cohomological and homotopic information of surfaces: using this information we decorate the smoothings of the standard Khovanov cube, before transferring the decorations into algebra. 
\end{abstract}

\maketitle

\section{Introduction}\label{Sec:introduction}
In this paper we construct additional gradings on two generalisations of Khovanov homology and use them to produce invariants of various kinds of embeddings. We use the functoriality of these generalisations with respect to knot cobordism to produce invariants of surfaces embedded in thickened \(3\)-manifolds. These invariants take the form of module maps, which decompose with respect to a grading derived from certain diagrams related to the thickened \(3\)-manifold into which the surface is embedded. As such they are refered to as \emph{picture-valued invariants}; \Cref{Fig:hgrad21,Fig:twopants,Fig:pantsgrids} (on \cpageref{Fig:hgrad21,Fig:twopants,Fig:pantsgrids}) contain an elucidation of these invariants, the full definition of which is contained in \Cref{Sec:homotopygrading}. These invariants are related to the theory of free knots and cobordisms between them; in contrast to the elementary methods used to study cobordisms of free knots in \cite{Fedoseev2017a,Fedoseev2017}, however, the methods of this paper are more involved.

The Khovanov homology of links in \( S^3 \) - henceforth known as \emph{classical links} - contains geometric information in the form of the Rasmussen invariant, powerful enough to answer complex questions like the Milnor Conjecture \cite{Rasmussen2010}. This geometric information can be employed to produce invariants of surfaces in \( B^4 \) \cite{Bar-natan2005,Jacobsson2004,Tanaka2005}.
 
While other homology theories of classical links - link Floer or instanton homology, for instance - are defined for links in \( 3 \)-manifolds other than \(S^3\), Khovanov homology does not naturally extend to such links. In this paper we focus on links in thickened surfaces and the closely-related virtual links, and use two extensions of Khovanov homology to such links to produce invariants of surfaces embedded in \(4\)-manifolds other than \(B^4\).

\subsection{Extending Khovanov homology}\label{Subsec:extending}
The fundamental problem one encounters when attempting to produce a Khovanov theory of general links in thickened surfaces (or virtual links) is that a new type of edge is present in the cube of resolutions. It is no longer guaranteed that the number of circles in a resolution will increase or decrease by \(1\) along an edge of the cube; it is possible that the number of circles remains constant. Edges linking resolutions with the same number of circles are known as \emph{single-cycle smoothings}. In order to preserve the quantum grading along single-cycle smoothings one must assign the zero map to them. This assignment does not yield a chain complex, however. 

There are at least two distinct methods of constructing a chain complex for virtual links. The first author developed a theory which uses twisted coefficients and other new technology in order to recover a chain complex \cite{Manturov2006}; this theory shall be referred to as \emph{virtual Khovanov homology}. This theory was reformulated by Dye, Kaestner, and Kauffman in order to define a \emph{virtual Rasmussen invariant} \cite{Dye2014}.

The second author altered the module associated to vertices of the cube of resolutions, so that the single-cycle smoothing may be assigned a non-zero map. The resulting theory is known as \emph{doubled Khovanov homology} \cite{Rushworth2017}. A perturbation similar to that of Khovanov homology defined by Lee \cite{Lee2005} exists for doubled Khovanov homology; using this perturbation a \emph{doubled Rasmussen invariant} can be defined.

Both the virtual and doubled Rasmussen invariants are slice obstructions for virtual knots (and hence knots in thickened surfaces), demonstrating that both virtual and doubled Khovanov homology contain geometric information.

In this paper we enhance the sensitivity of these homology theories to the higher dimensional information contained in a link in a thickened surface, using the topology of the surface itself. Specifically, we augment the theories using cohomological and homotopic information. While we do not produce new Rasmussen-style invariants from these enhanced theories, we do obtain new invariants of links in thickened surfaces, and invariants of surfaces in thickened \(3\)-manifolds, as advertised above.

\subsection{Plan of the paper}
The material contained in this paper is split into two distinct categories: the material of \Cref{Sec:gradingdef} is independent to that of \Cref{Sec:homotopygrading}, and the reader can read one without consulting the other. While both sections describe the construction of additional gradings, the homology theories used are distinct.

In \Cref{Sec:review} we recap material we require for the rest of the paper. This includes the construction of the cube of resolutions, and the two generalisations of Khovanov homology which we use to produce invariants. We also review some definitions regarding links in thickened surfaces and virtual links.

In \Cref{Sec:gradingdef} we describe the construction of an additional grading on the virtual and doubled Khovanov homology using the cohomology of surfaces, yielding a triply-graded homology theory. We also discuss two interacting spectral sequences which both begin at this triply-graded theory, and abut to non-isomorphic theories.

In \Cref{Sec:homotopygrading} we describe the construction of another, distinct, grading on virtual Khovanov homology. This grading is produced using homotopy classes of curves on surfaces; these homotopy classes are represented by diagrams, and as such the resulting grading is known as a \emph{picture-valued grading}.

In \Cref{Sec:cobordismmaps} we investigate how the additional gradings defined in the preceeding sections interact with cobordisms between links, before using them to define invariants in \Cref{Subsec:embeddedsurfaces}.

Throughout the paper \( \Sigma_g \) shall denote a closed compact connected oriented surface of genus \( g \). All embeddings are smooth. We also assume some familiarity with the definition of Khovanov homology for classical links (although \Cref{Subsec:classicalkhov} contains a brief overview), particularly in \Cref{Subsec:vkhreview,Subsec:dkhreview}. Finally, we shall use the labelling of virtual knots from Green's table \cite{Green}.

\section{Review}\label{Sec:review}
In this section we give a quick review of material required throughout the paper: the construction of Khovanov homology and two extensions of it used throughout, and some definitions regarding links in thickened surfaces and virtual links.

\subsection{Classical Khovanov homology}\label{Subsec:classicalkhov}
For concreteness we shall describe the construction of Khovanov homology of classical links, including the cube of resolutions and the algebraic complex.

Khovanov homology associates to an oriented classical link a bigraded finitely generated Abelian group \cite{Khovanov1999,Bar-Natan2014}. This is done in two steps. First, given an oriented link diagram \( D \), a formal chain complex is produced which is known as the \emph{cube of resolutions}. Next, this formal chain complex is converted into a algebraic chain complex by applying a topological quantum field theory i.e.\ unions of circles are assigned tensor products of modules, and the edges of the cube are assigned module maps. The chain homotopy equivalence class of the resulting chain complex depends only on the link, \( L \), represented by \( D \). Thus its homology is an invariant of \( L \), and is refered to as the \emph{Khovanov homology of \(L\)}.

We have illustrated these two steps below for a diagram of the Hopf link, where \( \mathcal{A} \) is a module over a ring \( \mathcal{R} \). Reading the figure from top to bottom we see the cube of resolutions associated to the diagram given on the left, then the algebraic complex, denoted \( \ckh \), and finally the Khovanov homology of the given link, denoted \( \kh \).

\begin{center}
	\begin{tikzpicture}[scale=0.7,
roundnode/.style={}]

\node[roundnode] (smooth)at (4,3.5) {\includegraphics[scale=0.5]{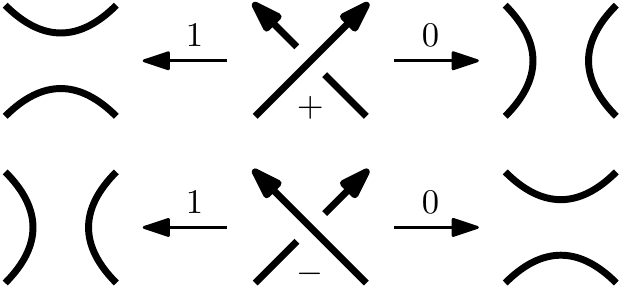}
};

\node[roundnode] (s4)at (-10,0)  {\includegraphics[scale=0.45]{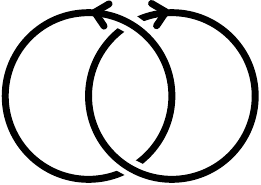}
};

\node[roundnode] (s0)at (-7,0)  {\(\begin{matrix}
	\includegraphics[scale=0.45]{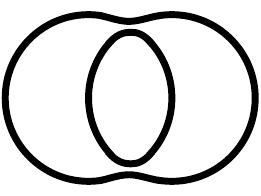} \\
	00
	\end{matrix} \)
};

\node[roundnode] (s1)at (-2,2)  {\(\begin{matrix}
	\includegraphics[scale=0.45]{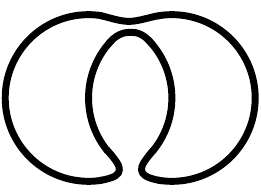} \\
	01
	\end{matrix} \)
};

\node[roundnode] (s2)at (-2,-2)  {\(\begin{matrix}
	\includegraphics[scale=0.45]{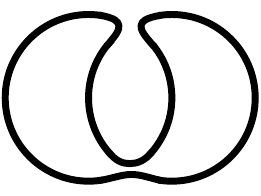} \\
	10
	\end{matrix} \)
};

\node[roundnode] (s3)at (3,0)  {\(\begin{matrix}
	\includegraphics[scale=0.45]{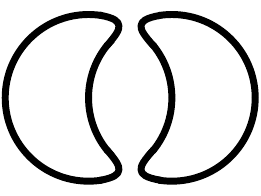} \\
	11
	\end{matrix} \)
};

\draw[->,double] (s4)--(s0) ;

\draw[->,thick] (s0)--(s1) node[above left,pos=0.6]{\( m \)} ;

\draw[->,thick] (s0)--(s2) node[below left,pos=0.6]{\( m \)} ;

\draw[->,thick] (s1)--(s3) node[above right,pos=0.4]{\( \Delta \)} ;

\draw[->,thick] (s2)--(s3) node[below right,pos=0.3]{\( -\Delta \)} ;

\node[roundnode] (s4)at (-10,-5)  {\( \ckh \left( \raisebox{-9pt}{\includegraphics[scale=0.45]{hopf.pdf}} \right) = \)
};

\node[roundnode] (s5)at (-7,-5) {\( \mathcal{A}^{\otimes 2} \)
};

\node[roundnode] (s6)at (-2,-5) {\( \begin{matrix}
	\mathcal{A} \\
	\oplus \\
	\mathcal{A} \\
	\end{matrix} \)
};

\node[roundnode] (s7)at (3,-5) {\( \mathcal{A}^{\otimes 2} \)
};

\draw[->,thick] (s5)--(s6) node[above,pos=0.5]{\small \( d_{-2} =  \begin{pmatrix}
	m \\
	m
	\end{pmatrix} \)
};

\draw[->,thick] (s6)--(s7) node[above,pos=0.5]{\small \( d_{-1} =  \left( \Delta, - \Delta \right) \)
};
\end{tikzpicture}
\end{center}

There is a distinguished basis of \( \ckh ( D ) \) given by \emph{states of \(D\)}: resolutions of \( D \) whose circles are decorated with either \( v_+ \) or \( v_- \). Given such a decorated smoothing, we denote by \( x = v_{ \pm } \otimes v_{ \pm } \otimes \cdots \otimes v_{ \pm } \) the element of \( \ckh ( D ) \) defined by it, and refer to both \( x \) and the decorated smoothing which defines it as a state of \( D \).

We define the bigrading of \( \ckh ( D ) \) on states. First, the \emph{homological grading}, denoted \( i \), is defined as
\begin{equation}
\label{Eq:classicalhom}
	i ( x ) = \# (1-\text{resolutions in the state}~x ) - n_-
\end{equation}
where \( n_- \) is the number of negative crossings of \( D \).

Next, the \emph{quantum grading}, denoted \( j \), is defined as
\begin{equation}
\label{Eq:classicalquant}
	j ( x ) = \# ( v_+\text{'s in the state}~x ) - \# ( v_-\text{'s in the state}~x ) + i ( x ) + wr(D)
\end{equation}
where \( wr(D) \) denotes the writhe of \(D\).

The homology of the chain complex \( \cdkh ( D ) \) is as follows; we have split the resulting group by its bigrading, with the horizontal axis denoting the homological grading, and the vertical axis the quantum grading:
\begin{center}
	\( \kh \left( \raisebox{-8pt}{\includegraphics[scale=0.45]{hopf.pdf}} \right) \quad = \quad	\raisebox{-80pt}{\begin{tikzpicture}[scale=0.8]
	
	%Axes
	\draw[black, ->] (-3.5,-4) -- (3.5,-4) ;
	\draw[black, ->] (-3.5,-4) -- (-3.5,1) ;
	
	\draw (-2,-3.8) -- (-2,-4.2) node[below] {\small $-2$} ;
	\draw (0,-3.8) -- (0,-4.2) node[below] {\small $-1$} ;
	\draw (2,-3.8) -- (2,-4.2) node[below] {\small $0$} ;
	
	\draw (-3.3,-3) -- (-3.7,-3) node[left] {\small $-4$} ;
	\draw (-3.3,-2) -- (-3.7,-2) node[left] {\small $-2$} ;
	\draw (-3.3,-1) -- (-3.7,-1) node[left] {\small $0$} ;
	\draw (-3.3,0) -- (-3.7,0) node[left] {\small $2$} ;
	
	%Homology
	\node[] (-2-3)at(-2,-3) {$ \mathcal{R} $} ;
	\node[] (-2-2)at(-2,-2) {$ \mathcal{R} $} ;
	
	\node[] (21)at(2,-1) {$ \mathcal{R} $} ;
	\node[] (23)at(2,0) {$ \mathcal{R} $} ;
	
	\end{tikzpicture}}
\)
\end{center}

\subsection{Virtual Khovanov homology}\label{Subsec:vkhreview}
As mentioned in \Cref{Subsec:extending}, virtual Khovanov homology proceeds by assigning the zero map to the single-cycle smoothing, while leaving the \(m\) and \( \Delta \) maps unchanged. One must then further modify the construction in order to recover a chain complex.

\begin{figure}
	\includegraphics[scale=1]{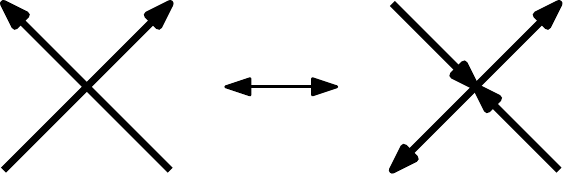}
	\caption{The source-sink decoration.}
	\label{Fig:sourcesink}
\end{figure}

Let \( \mathcal{A} = \mathcal{R}[X]/(X^2-t) \) for \( \mathcal{R} \) a commutative unital ring and \( t \in \mathcal{R} \). The key technique used in the construction of virtual Khovanov homology involves exploiting a symmetry present in \( \mathcal{A} \) (which corresponds to the two possible orientations of \( S^1 \)) using the following automorphism:

\begin{definition}\label{barring}
	The \emph{barring operator} is the \(\mathcal{R}\)-module map
	\begin{equation}
		\overline{\phantom{X}} : \mathcal{A} \rightarrow \mathcal{A},~ X \mapsto -X.
	\end{equation}
	Applying the barring operator is referred to as \emph{conjugation}.\CloseDef
\end{definition}

Note that if \( \mathcal{R} = \mathbb{R} \) and \( t = -1 \) then \( \mathcal{A} = \mathbb{C} \) and the barring operator is just standard complex conjugation. How the barring operator is applied within the Khovanov complex is determined using an extra decoration on link diagrams, the source-sink decoration as depicted in \Cref{Fig:sourcesink}. A new diagram is formed by replacing the classical crossings with the source-sink decoration, which induces an orientation on the incident arcs of a crossing. Loosely speaking, the barring operator is applied when the orientations induced by distinct crossings disagree. We refer the reader to \cite{Manturov2006,Dye2014}.
 
\subsection{Doubled Khovanov homology}\label{Subsec:dkhreview}
Doubled Khovanov homology provides an alternative extension of Khovanov homology to virtual links \cite{Rushworth2017}. The problem presented by the single-cycle smoothing is dealt with by ``doubling up'' the module assigned to a resolution; this allows the map assigned to the single-cycle smoothing, denoted \( \eta \), to be non-zero.

\begin{figure}
	\begin{tikzpicture}[scale=0.8]
	
	\node[] (s0)at (-5,0)  {
		\(\begin{matrix}
		0 \\
		v_+ \\
		0 \\
		v_-
		\end{matrix}\)};
	
	\node[] (s1)at (-3,0)  {
		\(\begin{matrix}
		v_+ \\
		0 \\
		v_- \\
		0
		\end{matrix}\)};
	
	\draw[->,thick] (s0)--(s1) node[above,pos=0.5]{\( \eta \)} ;
	
	\node[] (s2)at (3,0)  {
		\(\begin{matrix}
		0 \\
		\vup \\
		\vlp \\
		\vum \\
		\vlm
		\end{matrix}\)};
	
	\node[] (s3)at (5,0)  {
		\(\begin{matrix}
		\vup \\
		\vlp \\
		\vum \\
		\vlm \\
		0
		\end{matrix}\)};
	
	\draw[->,thick] (s2)--(s3) node[above,pos=0.5]{\( \eta \)} ;
	
	\end{tikzpicture}
	\caption{On the left, the complex of to the single-cycle smoothing if one assigns \( \mathcal{A} \) to a cycle. On the right, the complex of the single-cycle smoothing if one assigns \( \mathcal{A} \oplus \mathcal{A} \lbrace -1 \rbrace \) to a cycle. The generators are arranged vertically by quantum grading.}
	\label{Fig:doublingup}
\end{figure}

A schematic picture of this ``doubling up'' process is given in \Cref{Fig:doublingup}; the left hand complex depicts the situation when the module \( \mathcal{A} \) is assigned to a circle within a resolution. One sees immediately that the \( \eta \) map must be zero if it is to be degree-preserving. This is path followed by the first author and Dye et al, and outlined in the previous section. The right hand complex, however, depicts the situation arrived at if one assigns the module \( \mathcal{A} \oplus \mathcal{A} \lbrace -1 \rbrace \) to a circle, where \( \mathcal{A} = \langle \vup, \vum \rangle \) and \( \mathcal{A} \lbrace -1 \rbrace = \langle \vlp, \vlm \rangle \) (the superscripts are u for ``upper'' and l for ``lower'') \footnote{The module \( \mathcal{A} \) is graded (by the quantum grading), and for \( W \) a graded module \( W_{l-k} = {W \lbrace k \rbrace}_l\)}. This allows for \( \eta \) to be non-zero and degree preserving.

Given a virtual link diagram, \( D \), the complex \( \cdkh ( D ) \) is formed in the usual way: form the cube of resolutions of \( D \), then assign modules to the vertices and maps to the edges. The module assigned to a resolution of \( j \) circles is \( \mathcal{A}^{\otimes j} \oplus  \mathcal{A}^{\otimes j} \lbrace -1 \rbrace \). The maps \( m \) and \( \Delta \) familiar from classical Khovanov homology are essentially unchanged: they do not map between the upper and lower summands. Philosophically, one can think of the \( \eta \) map as ``intertwining'' the two summands together. We refer the reader to \cite{Rushworth2017}.

There is a perturbation of doubled Khovanov homology analogous to that defined by Lee in the classical case: adding a term to the differential of quantum degree \(+4\) yields a new homology theory, known as \emph{doubled Lee homology}. Unlike in the case of classical Lee homology, there exist virtual links which have trivial doubled Lee homology. The poses problems when using doubled Lee homology to investigate cobordism and concordance. The typical method is to employ the functoriality of the homology theory, and associate maps between homologies to cobordisms between links. One must take care to ensure that the homology of every virtual link appearing in a cobordism has non-trivial doubled Lee homology; otherwise, the map associated to the cobordism will be the zero map.

\subsection{Links in thickened surfaces and virtual links}\label{Subsec:thickenedreview}
Here we gather together some definitions regarding two closely related types of objects: links in thickened surfaces and virtual links. For more details we refer the reader to \cite{Turaev2007,Kauffman1998}.

\begin{definition}\label{Def:thickenedlink}
	A \emph{link in a thickened surface} is an embedding \( \bigsqcup S^1 \hookrightarrow \Sigma_g \times I \), considered up to isotopy. A \emph{link diagram on a thickened surface} (simply a \emph{diagram} if the usage is unambiguous) is a \(4\)-valent graph on \( \Sigma_g \), such that each vertex is decorated with the standard undercrossing or overcrossing; an example is given in \Cref{Fig:21lift}. A knot in a thickened surface is a one component link in a thickened surface.\CloseDef
\end{definition}

Two link diagrams represent the same link if and only if they are related by a finite sequence of Reidemeister moves (these moves are identical to those of classical knot theory). Given a link diagram on a thickened surface, it may be possible to remove all of its crossings via Reidemeister moves but not to convert it into a union of disjoint \emph{contractible} loops. For example, consider the meridian of a torus: it has no crossings but is not contractible (nor is any diagram related to it by Reidemeister moves). In what follows we shall refer to a link as \emph{crossingless} if it possesses a diagram with no crossings, and reserve the term \emph{unlink} for a link which possesses a diagram which is a collection of disjoint contractible simple curves. An \emph{unknot} is an unlink of one component.

\begin{definition}\label{Def:cobordism}
	Let \( K \hookrightarrow \Sigma_g \) and \( K' \hookrightarrow \Sigma_{g'} \) be knots in thickened surfaces. A \emph{cobordism from \( K \) to \( K' \)} is a compact connected oriented surface \( S \hookrightarrow M \times I \), where \( \partial S = K \sqcup K' \) and \( M \) is a compact oriented \(3\)-manifold such that \( \partial M = \Sigma_g \sqcup \Sigma_{g'} \). If the genus of \( S \) is zero we say that \( K \) and \( K' \) are \emph{concordant}.\CloseDef
\end{definition}

In what follows, we shall focus our attention on cobordisms which are embedded into \(4\)-manifolds of the form \( \Sigma_g \times I \times I \).

A \emph{virtual link} is an equivalence class of links in thickened surfaces up autodiffeomorphism of the surface and handle stabilisations whose attaching spheres do not intersect the link. They can be represented by \emph{virtual link diagrams}: link diagrams with an extra crossing decoration, the virtual crossing \raisebox{-2.5pt}{\includegraphics[scale=0.3]{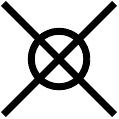}}. An example is given in \Cref{Fig:21lift}. A virtual link is an equivalence class of virtual link diagrams up to the virtual Reidemeister moves; these moves contain the three classical Reidemeister moves, plus four new moves involving virtual crossings.

\Cref{Def:cobordism} applies to virtual knots also; in particular, it is clear that two representatives of a virtual knot as a knot in a thickened surface are concordant.

\section{Gradings from cohomology}\label{Sec:gradingdef}
In this section we describe the construction of additional gradings on doubled Khovanov homology of a link in a thickened surface, using the first cohomology of the surface to produce the grading. For the definition in the case of virtual Khovanov homology we refer the reader to \cite{Manturov2008a}.

In \Cref{Sec:cobordismmaps} we describe the assignment of maps on homology to cobordisms between links, and in \Cref{Subsec:dotteddoubledinvariants} use them to obtain an obstruction to knots in \( \Sigma_g \times I \) bounding a disc in \( \Sigma_g \times I \times I \).

\subsection{Decorating the cube of resolutions}\label{Subsec:decorating}
Conceptually, we wish to make use of the way in which an embedding \( \bigsqcup S^1 \hookrightarrow \Sigma_g \times I \) is knotted around \( \Sigma_g \) in order to enchance the standard doubled Khovanov complex. Doubled Khovanov homology is quite insensitive to this information, as evidenced by its invariance under the purely virtual Reidemeister moves and the move known as flanking \cite{Rushworth2017}.

We begin by decorating the cube of resolutions in a manner which captures some of this information.

\begin{definition}[The dotted cube of resolutions]
	\label{Def:dottedcube}
	Let \( D \) be a diagram of an oriented link in a thickened surface \( L \hookrightarrow \Sigma_g \times I \). Form the cube of resolutions of \( D \) in the same way as for a virtual link diagram: resolutions are embeddings of disjoint unions of \( S^1 \) into \( \Sigma_g \). An example is given in \Cref{Fig:d21cube}.
	
	Pick an element \( c \in H^1 ( \Sigma_g ; \Z_2) \). Decorate the cube of resolutions as follows: a circle within a resolution is decorated with a \emph{dot} if it has non-zero image under \( c \). The assignment of dots to all circles of all resolutions withing the cube is referred to as the \emph{dotting associated to \( c \)}.
	
	Two examples of dottings are given in \Cref{Fig:d21cube}; green dots represent the dotting associated to the element of \( H^1 ( \Sigma_g ; \Z_2) \) coloured green, and the element coloured red does not produce any dots. The fully decorated cube is refered to as the \emph{dotted cube of resolutions of \( D \) with respect to \( c \)}, denoted \( \llbracket D, h \rrbracket \).\CloseDef
\end{definition}

Of course, the dotted cube of resolutions defined above depends on the choice of \( c \in H^1 ( \Sigma_g ; \Z_2) \). We shall show that the resulting homology is an invariant of the pair \( (L, c) \).

\begin{figure}
	\includegraphics[scale=0.75]{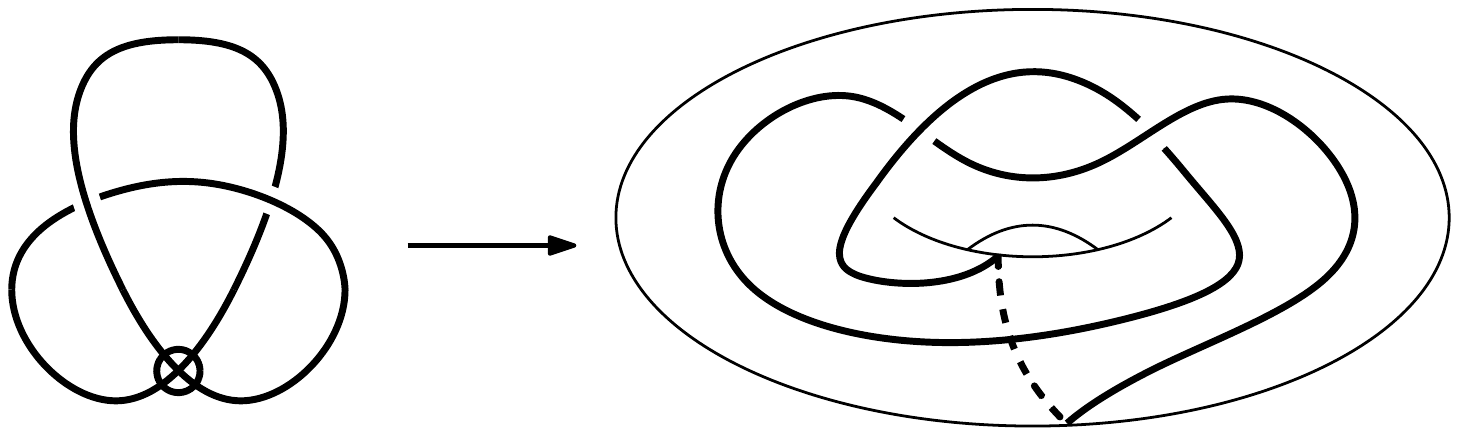}
	\caption{A knot in \( \Sigma_1 \times I \) which is a lift of virtual knot \( 2.1 \).}
	\label{Fig:21lift}
\end{figure}

\begin{figure}
	\begin{tikzpicture}[scale=1,
	roundnode/.style={}]
	
	\node[roundnode] (s0)at (-6.5,0)  {\includegraphics[scale=0.5]{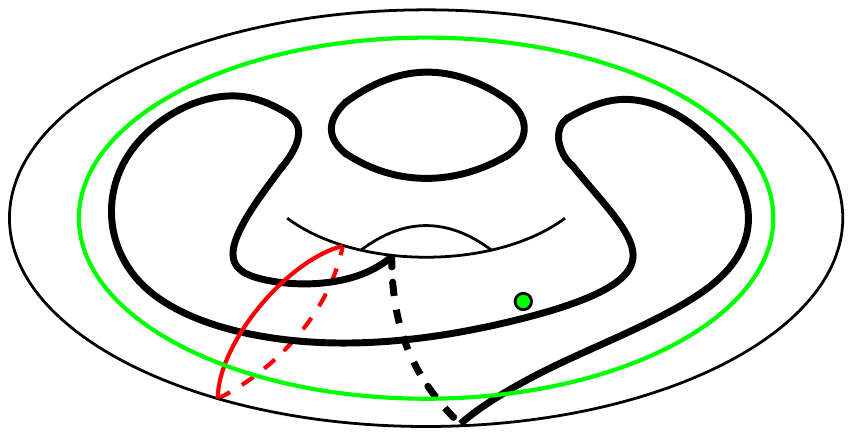}
	};
	
	\node[roundnode] (s1)at (-2,2.75)  {\includegraphics[scale=0.5]{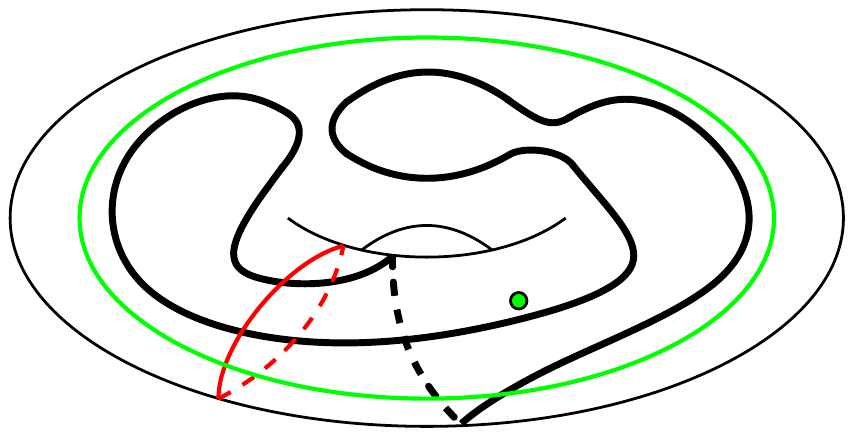}
	};
	
	\node[roundnode] (s2)at (-2,-2.75)  {\includegraphics[scale=0.5]{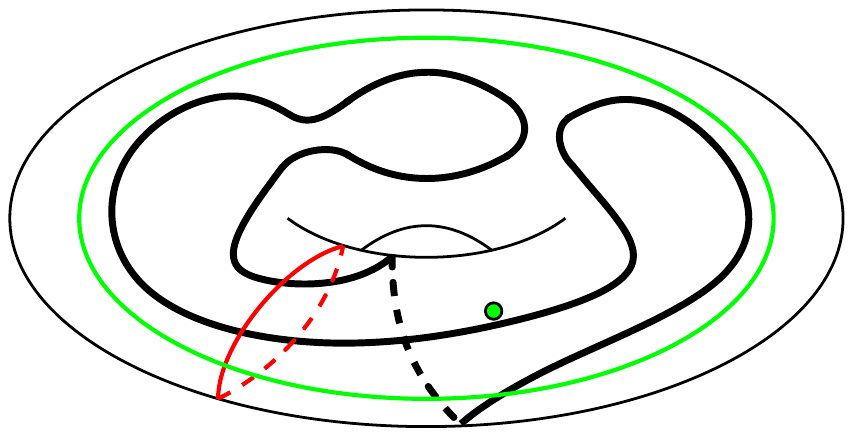}
	};
	
	\node[roundnode] (s3)at (2.5,0)  {\includegraphics[scale=0.5]{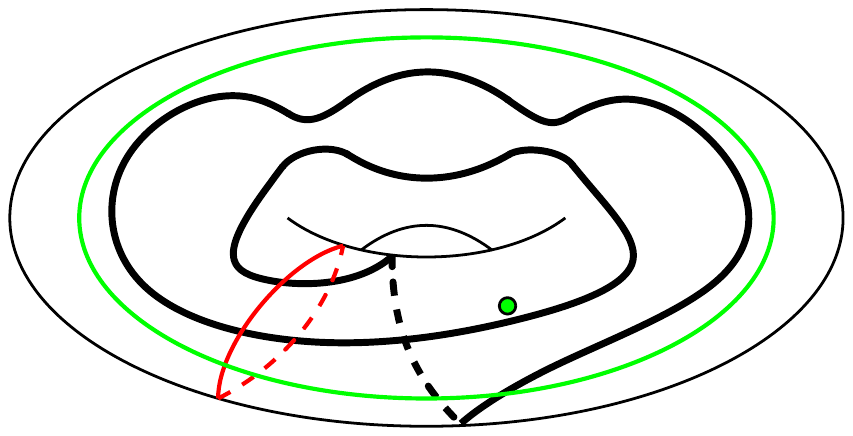}
	};
	
	\draw[->,thick] (s0)--(s1) node[above left,pos=0.6]{\( m \)} ;
	
	\draw[->,thick] (s0)--(s2) node[below left,pos=0.6]{\( m \)} ;
	
	\draw[->,thick] (s1)--(s3) node[above right,pos=0.4]{\( \eta \)} ;
	
	\draw[->,thick] (s2)--(s3) node[below right,pos=0.3]{\( -\eta \)} ;
	\end{tikzpicture}
	\caption{The dotted cube of resolutions of the diagram depicted in \Cref{Fig:21lift}.}
	\label{Fig:d21cube}
\end{figure}

\subsection{Doubled Khovanov homology with dots}\label{Subsec:dotteddkk}
We now incorporate the higher dimensional information, in the form of dots, into the algebraic complex. We keep track of the dots of the resolutions by placing a dot above the module associated to a dotted circle.

\begin{definition}[The dotted doubled Khovanov complex]
	\label{Def:dottedcomplex}
	Let \( D \) be a diagram of an oriented link in a thickened surface \( L \hookrightarrow \Sigma_g \times I \). Pick an element \( c \in H^1 ( \Sigma_g ; \Z_2) \) and form the dotted cube \( \llbracket D, c \rrbracket \) as in \Cref{Def:dottedcube}. We form the \emph{doubled Khovanov complex of \(D\) with respect to \(c\)} in the manner as outlined in \Cref{Subsec:dkhreview}, but augmented by adding dots above modules assigned to circles which are dotted. These dots persist to elements of the dotted module; that is, we denote the elements of \( \overset{\bullet}{\mathcal{A}} \) as \( v_{\dpl} \) and \( v_{\dm} \).
	
	As in unaugmented Khovanov homology, the components of the differential are matrices of the appropriate maps, which are assigned signs in the standard way. The resulting chain complex is denoted \( \cdkh ( D, c ) \), and an example of such a complex is given in \Cref{Fig:dottedcomp}.\CloseDef
\end{definition}

\begin{definition}
	\label{Def:hgrading}
	Let \( D \) be a diagram of an oriented link in a thickened surface \( L \hookrightarrow \Sigma_g \times I \) and \( \cdkh  ( D, c ) \) its dotted doubled Khovanov complex with respect to \( c \in H^1 ( \Sigma_g ; \Z_2) \). By an abuse of notation we denote by \( c \) both the cohomology class and a \( \Z [ \frac{1}{2}] \)-grading on \( \cdkh  ( D, c ) \) defined in the following manner. Given \( x \in \overset{ \left( \bullet \right) }{\mathcal{A}} \otimes \overset{ \left( \bullet \right) }{\mathcal{A}} \otimes \cdots \otimes \overset{ \left( \bullet \right) }{\mathcal{A}} \) (where the copies of \( \mathcal{A} \) may or may not be dotted) define
		\begin{equation}
			c ( x ) \coloneqq \# \left( v_{\dm} \right) - \# \left( v_{\dpl} \right) + \frac{1}{2} j ( x )
		\end{equation}
	where \( \# \left( v_{\dpl} \right) \) denotes the number of \( v_{\dpl} \) in \( x \) (likewise \( \# \left( v_{\dm} \right) \) the number of \( v_{\dm} \) ), and \( j \) the standard quantum degree. We refer to this grading as the \emph{\( c \)-grading}.\CloseDef
\end{definition}

Of course, the \( c \)-grading contains no new information if \( c \) is a trivial cohomology class or if no circles within the cube of resolutions are assigned dots (see \Cref{Prop:trivialdots}).

\begin{figure}
	\begin{tikzpicture}[scale=0.6,
	roundnode/.style={}]
	\node[roundnode] (s5)at (-8.5,-7) {\( \begin{matrix}
		\mathcal{A} \otimes \overset{{\color{green} \bullet}}{\mathcal{A}} \\
		\oplus \\
		\mathcal{A} \otimes \overset{{\color{green} \bullet}}{\mathcal{A}} \lbrace -1 \rbrace \\
		\end{matrix} \)
	};
	
	\node[roundnode] (s6)at (-1.5,-7) {\( \begin{matrix}
		\overset{{\color{green} \bullet}}{\mathcal{A}} \\
		\oplus \\
		\overset{{\color{green} \bullet}}{\mathcal{A}} \lbrace -1 \rbrace \\
		\oplus \\
		\overset{{\color{green} \bullet}}{\mathcal{A}} \\
		\oplus \\
		\overset{{\color{green} \bullet}}{\mathcal{A}} \lbrace -1 \rbrace
		\end{matrix} \)
	};
	
	\node[roundnode] (s7)at (5.5,-7) {\( \begin{matrix}
		\overset{{\color{green} \bullet}}{\mathcal{A}} \\
		\oplus \\
		\overset{{\color{green} \bullet}}{\mathcal{A}} \lbrace -1 \rbrace \\
		\end{matrix} \)
	};
	
	\node[roundnode] (s8)at(-8.5,-10.5) {\( -2 \)};
	
	\node[roundnode] (s9)at(-1,-10.5) {\( -1 \)};
	
	\node[roundnode] (s10)at(5.5,-10.5) {\( 0 \)};
	
	\draw[->,thick] (s5)--(s6) node[above,pos=0.5]{\( d_{-2} =  \begin{pmatrix}
		m \\
		m
		\end{pmatrix} \)
	};
	
	\draw[->,thick] (s6)--(s7) node[above,pos=0.5]{\( d_{-1} =  \left( \eta, - \eta \right) \)
	};
	\end{tikzpicture}
	\caption{The dotted complex associated to the cube of resolutions depicted in \Cref{Fig:d21cube}, with respect to the green cohomology class.}
	\label{Fig:dottedcomp}
\end{figure}

The components of the differential of \( \cdkh ( D, c ) \) split with respect to the \( c \)-grading as
\begin{equation*}
	d = d^0 + d^{+2}
\end{equation*}
where \( d^i \) is \(c\)-graded of degree \(i\). For the \( m \), \( \Delta \), and \( \eta \) maps, the particular splitting depends on the configuration of the dots assigned to the circles involved. We shall denote by \( \underline{\phantom{X}} \rightarrow \bullet \otimes \bullet \) a \( \Delta \) map taking an undotted circle to two dotted circles, and so forth. We have suppressed the \( \text{u} / \text{l} \) superscripts for the \( m \) and \( \Delta \) maps (as they do not interact with them). (Terms in parentheses denote components of the differential of doubled Lee homology, and are required later.)

\begin{equation}\label{Eq:dkkdiff1}
	m : \bullet \otimes \bullet \rightarrow \underline{\phantom{X}} \left\lbrace
	\begin{aligned}
		v_{\dpl \dpl} & \xrightarrow{m^{0}} 0 \qquad & v_{\dpl \dpl} & \xrightarrow{m^{+2}} v_+ \\
		v_{\dpl \dm}, v_{\dm \dpl} & \xrightarrow{m^{0}} v_- \qquad & v_{\dpl \dm}, v_{\dm \dpl} & \xrightarrow{m^{+2}} 0 \\
		v_{\dm \dm} & \xrightarrow{m^{0}} 0 ~ ( v_+ ) \qquad & v_{\dm \dm} & \xrightarrow{m^{+2}} 0
	\end{aligned}\right.	
\end{equation}

\begin{equation}\label{Eq:dkkdiff2}
	m : \bullet \otimes \underline{\phantom{X}} \rightarrow \bullet \left\lbrace
	\begin{aligned}
		v_{\dpl +} & \xrightarrow{m^{0}} v_{\dpl} \qquad & v_{\dpl +} & \xrightarrow{m^{+2}} 0 \\
		v_{\dpl -} & \xrightarrow{m^{0}} 0 \qquad & v_{\dpl -} & \xrightarrow{m^{+2}} v_{\dm} \\
		v_{\dm +} & \xrightarrow{m^{0}} v_{\dm} \qquad & v_{\dm +} & \xrightarrow{m^{+2}} 0 \\
		v_{\dm -} & \xrightarrow{m^{0}} 0 ~ ( v_{\dpl} ) \qquad & v_{\dm \dm} & \xrightarrow{m^{+2}} 0
	\end{aligned}\right.	
\end{equation}
In the case \( m : \underline{\phantom{X}} \otimes \underline{\phantom{X}} \rightarrow \underline{\phantom{X}} \) we have \( m^{+2} = 0 \) so that \( m^0 = m \).

\begin{equation}\label{Eq:dkkdiff3}
	\Delta : \bullet \rightarrow \bullet \otimes \underline{\phantom{X}} \left\lbrace
	\begin{aligned}
		v_{\dpl} & \xrightarrow{\Delta^{0}} v_{\dpl -} \qquad & v_{\dpl} & \xrightarrow{\Delta^{+2}} v_{\dm +} \\
		v_{\dm} & \xrightarrow{\Delta^{0}} v_{\dm -} ~ ( + v_{\dpl +} ) \qquad & v_{\dm} & \xrightarrow{\Delta^{+2}} 0
	\end{aligned}\right.	
\end{equation}

\begin{equation}\label{Eq:dkkdiff4}
	\Delta : \underline{\phantom{X}} \rightarrow \bullet \otimes \bullet \left\lbrace
	\begin{aligned}
		v_{+} & \xrightarrow{\Delta^{0}} v_{\dpl \dm} + v_{\dm \dpl} \qquad & v_{+} & \xrightarrow{\Delta^{+2}} 0 \\
		v_{-} & \xrightarrow{\Delta^{0}} 0 ~ ( v_{\dpl \dpl} ) \qquad & v_{-} & \xrightarrow{\Delta^{+2}} v_{\dm \dm}
	\end{aligned}\right.	
\end{equation}
Again, for \( \Delta : \underline{\phantom{X}} \rightarrow \underline{\phantom{X}} \otimes \underline{\phantom{X}} \) we have \( \Delta^{+2} = 0 \).

\begin{equation}\label{Eq:dkkdiff5}
	\eta : \bullet \rightarrow \bullet \left\lbrace
	\begin{aligned}
		v^{\text{u}}_{\dpl} & \xrightarrow{\eta^{0}} v^{\text{l}}_{\dpl} \qquad & v^{\text{u}}_{\dpl} & \xrightarrow{\eta^{+2}} 0 \\
		v^{\text{l}}_{\dpl} & \xrightarrow{\eta^{0}} 0 \qquad & v^{\text{l}}_{\dpl} & \xrightarrow{\eta^{+2}} 2 v^{\text{u}}_{\dm} \\
		v^{\text{u}}_{\dm} & \xrightarrow{\eta^{0}} v^{\text{l}}_{\dm} \qquad & v^{\text{u}}_{\dm} & \xrightarrow{\eta^{+2}} 0 \\
		v^{\text{l}}_{\dm} & \xrightarrow{\eta^{0}} 0 ~ ( 2 v^{\text{u}}_{\dpl} ) \qquad & v_{\dm \dm} & \xrightarrow{\eta^{+2}} 0
	\end{aligned}\right.	
\end{equation}
Finally, for \( \eta : \underline{\phantom{X}} \rightarrow \underline{\phantom{X}} \) we have \( \eta^{+2} = 0 \), as usual.

\begin{theorem}
	\label{Thm:dothominvariance}
	Let \( D \) be a diagram of an oriented link in a thickened surface \( L \hookrightarrow \Sigma_g \times I \) and \( \cdkh  ( D, c ) \) its dotted doubled Khovanov complex with respect to \( c \in H^1 ( \Sigma_g ; \Z_2) \). The homology of \( \cdkh ( D, c ) \) with respect to \( d^0 \), the \(c\)-grading preserving component of the differential, is well-defined and is an invariant of \( L \). It is denoted by \( \dkh ( L, c ) \), and refered to as the \emph{doubled Khovanov homology of \( L \) with respect to \(c\)}.
\end{theorem}

\begin{proof}
	We need only show that \( \dkh ( L , c ) \) is invariant under the Reidemeister moves. Consider the two cubes of resolutions associated to the two tangle diagrams involved in a Reidemeister move; in order for \( \dkh ( L , c ) \) to be an invariant we must have that any circles appearing in the cubes are not assigned dots. This holds as they are necessarily homologically trivial. We can then apply the now standard Bar-Natan proof \cite{Bar-Natan2002,Manturov2013}; this hinges on the fact that \( m \) is surjective and \( \Delta \) injective. For full details we refer the reader to \cite[Section \(4\)]{Manturov2008a}.
\end{proof}

An example of the resulting homology is given in \Cref{Fig:21dottedhom}.

\begin{figure}
	\( \dkh \left( ~ \raisebox{-15pt}{\includegraphics[scale=0.3]{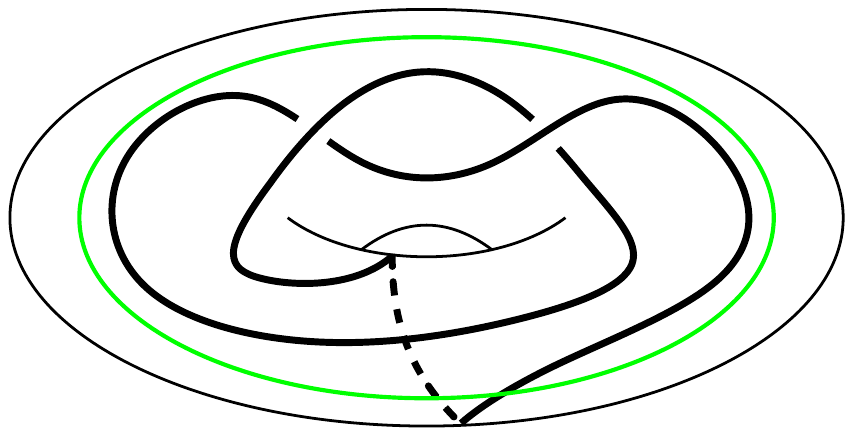}} ~, \raisebox{-5pt}{\includegraphics[scale=0.5]{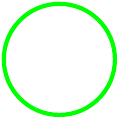}} ~ \right) \quad = \quad	\raisebox{-80pt}{\begin{tikzpicture}[scale=0.8]
		
		%Axes
		\draw[black, ->] (-3.5,-4) -- (3.5,-4) ;
		\draw[black, ->] (-3.5,-4) -- (-3.5,4) ;
		
		\draw (-2,-3.8) -- (-2,-4.2) node[below] {\small $-2$} ;
		\draw (0,-3.8) -- (0,-4.2) node[below] {\small $-1$} ;
		\draw (2,-3.8) -- (2,-4.2) node[below] {\small $0$} ;
		
		\draw (-3.3,-3) -- (-3.7,-3) node[left] {\small $-7$} ;
		\draw (-3.3,-2) -- (-3.7,-2) node[left] {\small $-6$} ;
		\draw (-3.3,-1) -- (-3.7,-1) node[left] {\small $-5$} ;
		\draw (-3.3,0) -- (-3.7,0) node[left] {\small $-4$} ;
		\draw (-3.3,1) -- (-3.7,1) node[left] {\small $-3$} ;
		\draw (-3.3,2) -- (-3.7,2) node[left] {\small $-2$} ;
		\draw (-3.3,3) -- (-3.7,3) node[left] {\small $-1$} ;
		
		%Homology
		\node[] (-2-3)at(-2,-3) {$ -\dfrac{5}{2} $} ;
		\node[] (-2-2)at(-2,-2) {$ -2 $} ;
		\node[] (-2-1')at(-2,-1) {$ -\dfrac{3}{2} \oplus -\dfrac{7}{2} $} ;
		\node[] (-20)at(-2,-0) {$ -3 $} ;
		
		\node[] (0-1)at(0,-1) {$ -\dfrac{3}{2}$} ;
		\node[] (0-1')at(0,1) {$ - \dfrac{5}{2} $} ;
		
		\node[] (21)at(2,1) {$ -\dfrac{1}{2} $} ;
		\node[] (23)at(2,3) {$ -\dfrac{3}{2} $} ;
		
		\end{tikzpicture}}
	\)
	
	\caption{The homology of the complex given in \Cref{Fig:dottedcomp}. The horizontal (vertical) axis denotes the homological (quantum) grading, and the terms on the grid points denote copies of \( \Z \) which generate the homology, along with their \(c\)-grading.}
	\label{Fig:21dottedhom}
\end{figure}

Of course, if the dotting associated to a cohomology class is trivial the resulting homology contains no new information.

\begin{proposition}\label{Prop:trivialdots}
	Let \( D \) be a diagram of an oriented link in a thickened surface \( L \hookrightarrow \Sigma_g \times I \). Suppose that \( c \in H^1 ( \Sigma_g ; \Z_2 ) \) is homologically trivial, or is such that no circle in any resolution in \( \llbracket D, c \rrbracket \) is assigned a dot. Then \( \dkh ( L, c ) = \dkh ( L )  \). That is, the doubled Khovanov homology of \( L \) with respect to \( c \) is equal to its unaugmented homology. This holds also for doubled Lee homology.
\end{proposition}

\subsection{Spectral sequences}\label{Subsec:spectralsequences}
Lee defined a spectral sequence on Khovanov homology whose \( E_\infty \) page is known as Lee homology. Rasmussen used Lee homology to define the \(s\)-invariant of classical knots. There is a similar spectral sequence on doubled Khovanov homology also, and the \( E_\infty \) page is doubled Lee homology \cite{Rushworth2017}. These spectral sequences are defined by adding extra terms to the differential. We can use this technique with respect to the \(c\)-grading defined in \Cref{Subsec:dotteddkk} also.

First, consider the homology with respect to the differential which includes the terms in parentheses in \Cref{Eq:dkkdiff1,Eq:dkkdiff2,Eq:dkkdiff3,Eq:dkkdiff4,Eq:dkkdiff5}. These terms raise the quantum grading by \(4\).

\begin{theorem}\label{Thm:leedotinvariance}
	Denote by \( {\widetilde{d}}^0 \) the differential obtained by adding the terms in parentheses (in \Cref{Eq:dkkdiff1,Eq:dkkdiff2,Eq:dkkdiff3,Eq:dkkdiff4,Eq:dkkdiff5}) to \( d^0 \). Let \( \cdkh ' ( D, c ) \) be the chain complex with chain spaces equal to \( \cdkh ( D, c ) \) but with differential given by \( {\widetilde{d}}^0 \). The homology of \( \cdkh ' ( D, c ) \) with respect to \( {\widetilde{d}}^0 \) is an invariant of the link represented by \( D \), and is denoted \( \dkh ' ( L, c ) \). We refer to \( \dkh ' ( L, c ) \) as the \emph{doubled Lee homology of \(L\) with respect to \(c\)}.
\end{theorem}

This homology is filtered with respect to the quantum grading, but graded with respect to the \(c\)-grading. Next, we introduce a filtration of the \(c\)-grading also.

\begin{definition}\label{Def:spectral}
	Let \( D \) be a diagram of a link in a thickened surface \( L \hookrightarrow \Sigma_g \times I \), and pick \( c \in H^1 ( \Sigma_g ; \Z_2 ) \). Let \( \cdkh '' ( D , c ) \) be the chain complex with chain spaces equal to \( \cdkh ' ( D , c ) \), and differential obtained from \( \tilde{d}^0 \) by adding \( d^{+2} \) (defined in \Cref{Eq:dkkdiff1,Eq:dkkdiff2,Eq:dkkdiff3,Eq:dkkdiff4,Eq:dkkdiff5}). Denote this differential by \( d'' \), and define \( \dkh '' ( L , c ) \) to be the homology of \( \cdkh '' ( D , c ) \) with respect to it.\CloseDef
\end{definition}

\begin{theorem}
	\label{Thm:doubleprimeinvariance}
	The homology \( \dkh '' ( L , c ) \) is an invariant of \( L \), and is refered to as the \emph{totally reduced homology of \( L \) with respect to \(c\)}.
\end{theorem}

\begin{proof}
	A diagram of a link in a thickened surface projects to a diagram of virtual link. It is easy to see that the chain complex \( \cdkh '' ( D , c ) \) is equal to the standard doubled Lee complex associated to the virtual link diagram to which \( D \) projects to; we have added the components which raise the \(c\)-grading, recovering the full differential. That doubled Lee homology is invariant under the virtual Reidemeister moves shows that \( \dkh '' ( L , c ) \) is invariant also. 
\end{proof}

\begin{corollary}\label{Cor:doubleprimelee}
	Let \( L \hookrightarrow \Sigma_g \times I \) be a link in a thickened surface. Forgetting the \(c\)-grading \( \dkh '' ( L , c ) \cong \dkh ' ( \underline{L} ) \), where \( \underline{L} \) denotes the virtual link represented by \( L \). 
\end{corollary}

As a result of \Cref{Cor:doubleprimelee} we see that, ignoring the \(c\)-grading, \( \dkh '' \) behaves identically to the doubled Lee homology of virtual links. An important trait of doubled Lee homology is that its rank is determined by the number of \emph{alternately coloured resolutions} of the argument link; thus it is possible for \( \dkh '' \) to vanish (for more details regarding such resolutions we refer the reader to \cite[Section \(3\)]{Rushworth2017}). When using \( \dkh '' \) to define invariants of surfaces we must take care of this phenomena, as is done in \cite[Section \(3.2\)]{Rushworth2017}.

\begin{figure}
	\( \dkh '' \left( ~ \raisebox{-32pt}{\includegraphics[scale=0.45]{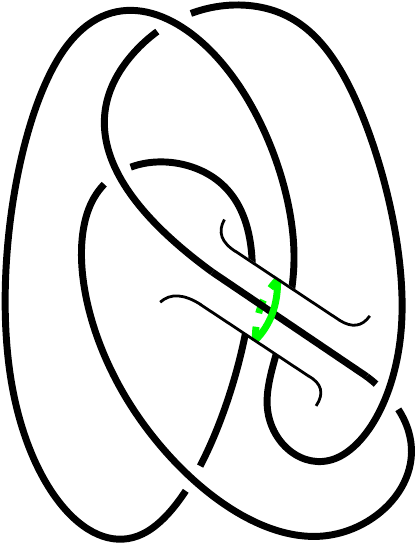}} ~, \raisebox{-5pt}{\includegraphics[scale=0.5]{unknotgreen.pdf}} ~ \right) \quad = \quad	\raisebox{-80pt}{\begin{tikzpicture}[scale=0.8]
		
		%Axes
		\draw[black, ->] (-3,-4) -- (2,-4) ;
		\draw[black, ->] (-3,-4) -- (-3,1) ;
		
		\draw (-2,-3.8) -- (-2,-4.2) node[below] {\small $-1$} ;
		\draw (-1,-3.8) -- (-1,-4.2) node[below] {\small $-\frac{1}{2}$} ;
		\draw (0,-3.8) -- (0,-4.2) node[below] {\small $0$} ;
		\draw (1,-3.8) -- (1,-4.2) node[below] {\small $\frac{1}{2}$} ;
		
		\draw (-2.8,-3) -- (-3.2,-3) node[left] {\small $-2$} ;
		\draw (-2.8,-2) -- (-3.2,-2) node[left] {\small $-1$} ;
		\draw (-2.8,-1) -- (-3.2,-1) node[left] {\small $0$} ;
		\draw (-2.8,0) -- (-3.2,0) node[left] {\small $1$} ;
		
		%Homology
		
		\node[] ()at(0,-3) {$ \Q $} ;
		
		\node[] ()at(1,-2) {$ \Q $} ;
		
		\node[] ()at(-2,-1) {$ \Q $} ;
		
		\node[] ()at(-1,0) {$ \Q $} ;
		
		\end{tikzpicture}}
	\)	
	\caption{The totally reduced homology of the knot depicted with respect to the generator depicted in green; the knot is a lift of the virtual knot \(4.12\). All of the generators are at homological grading \(0\), and the horizontal (vertical) axis denotes the \(c\)-grading (quantum grading).}
	\label{Fig:412dottedprimehom}
\end{figure}

Of course, we could have added the terms in \Cref{Eq:dkkdiff1,Eq:dkkdiff2,Eq:dkkdiff3,Eq:dkkdiff4,Eq:dkkdiff5} which raise the \(c\)-grading first, then added the terms in parentheses, to arrive at the totally reduced homology. In other words, there exists a commutative square of spectral sequences:
\begin{center}
	\begin{tikzpicture}
		\node[] (0)at(0,0) { \( \dkh ( L , c ) \) } ;
		\node[] (1)at(2,-1) { \( \dkh' ( L , c ) \) } ;
		\node[] (2)at(-2,-1) { \( \mathcal{H} ( L , c ) \) } ;
		\node[] (3)at(0,-2) { \( \dkh'' ( L , c ) \) } ;
		
		\draw[->] (0) -- (1) ;
		\draw[->] (0) -- (2) ;
		\draw[->] (1) -- (3) ;
		\draw[->] (2) -- (3) ;
	\end{tikzpicture}
\end{center}
where \( \mathcal{H} ( L , c ) \) is the homology obtained from \( \dkh ( L , c ) \) by adding the terms labelled \( d^{+2} \) in \Cref{Eq:dkkdiff1,Eq:dkkdiff2,Eq:dkkdiff3,Eq:dkkdiff4,Eq:dkkdiff5}. We conclude by remarking that the groups \( \mathcal{H} ( L , c ) \) and \( \dkh' ( L , c ) \) are mysterious; understanding their structure is an interesting direction of further research.

\subsection{Interaction with cobordisms}\label{Sec:cobordismmaps}
In this section we describe the process of associating maps between homology groups to cobordisms between links in thickened surfaces.

\subsubsection{Elementary cobordisms}\label{Subsec:elementary}
As in the classical case, cobordisms between knots in thickened surfaces can be analysed by decomposing them into simple pieces, which have \(0\) or \(1\) Morse critical points.

\begin{definition}\label{Def:elementary}
	Let \( L \hookrightarrow \Sigma_g \times I \) be a link in a thickened surface. Consider the following descriptions of a cobordism, \( S \), which begins at \( L \) and ends at another (possibly distinct) link in \( \Sigma_g \times I \):
	\begin{enumerate}[(i)]
		\item \( S \) is described by a single Reidemeister move on \( L \)
		\item \( S \) is described by the birth of a contractible loop in \( \Sigma_g \). This is known as a \emph{\(0\)-handle addition}.
		\item \( S \) is described by the death of a contractible loop in \( \Sigma_g \). This is known as a \emph{\(2\)-handle addition}.
		\item \( S \) is described by an oriented saddle on \( L \). This is known as a \emph{\(1\)-handle addition}.
		\item\label{Item:elementarystab} \( S \) is described by a single stabilisation on \( \Sigma_g \).
	\end{enumerate}
	The cobordisms described above are known as \emph{elementary cobordisms}.\CloseDef
\end{definition}

It is clear that any cobordism can be written as a concatenation of elementary cobordisms. We shall be focusing on cobordisms embedded in \(4\)-manifolds of the form \( \Sigma_g \times I \times I \), i.e.\ those that can be described without elementary cobordisms of the form given in \( ( \text{v} ) \) above.

\subsubsection{Maps on \( \dkh ( L, c) \), \( \dkh ' ( L, c) \), and \( \dkh '' ( L, c) \) }\label{Subsec:dkhprimemaps}
Given \( L \hookrightarrow \Sigma_g \times I \), \( c \in H^1 ( \Sigma_g ; \Z_2 ) \), the three theories \( \dkh ( L, c) \), \( \dkh ' ( L, c) \), and \( \dkh '' ( L, c) \) are nested i.e.\ there is the relationship:
\begin{center}
	\begin{tikzpicture}[roundnode/.style={}]
	\node[roundnode] (s0)at (-5.75,0)  {\( \dkh ( L,c ) \)
	};
	
	\node[roundnode] (s1)at (0,0)  {\( \dkh ' ( L,c ) \)
	};
	
	\node[roundnode] (s2)at (5.75,0)  {\( \dkh '' ( L,c ) \)
	};
	
	\draw[->] (s0)--(s1) node[above,pos=0.5]{filtration in \(j\)-grading} node[below,pos=0.5]{add Lee components} ;
	
	\draw[->] (s1)--(s2) node[above,pos=0.5]{filtration in \(c\)-grading} node[below,pos=0.5]{add \(d^{+2}\) components};
	
	\end{tikzpicture}
\end{center}
In light of this, we will suffice ourselves by describing the process of assigning maps on \( \dkh \) to cobordisms, as the process is identical for \( \dkh ' \) and \( \dkh '' \) modulo taking the appropriate filtration and adding the appropriate components to the maps assigned to \( 1 \)-handles (the maps assigned to \(0\)- and \(2\)-handles remain the same).

Let \( L, L' \hookrightarrow \Sigma_g \times I \) and \( S \hookrightarrow \Sigma_g \times I \times I \) an elementary cobordism such that \( \partial S = L \sqcup L' \) (this precludes \(S\) being of type \( (\text{v}) \) in \Cref{Def:cobordism}). Given \( c \in H^1 ( \Sigma_g ; \Z_2 ) \) we assign a map \( \phi_S : \dkh ( L, c ) \rightarrow \dkh ( L', c ) \) to \( S \) depending on its type. This assignment is made in essentially identical fashion to that described in \cite[Section \(3.2\)]{Rushworth2017}; for completeness, we shall give the form of the maps assigned to \(0\)-, \(1\)-, and \(2\)-handle additions.

Let \( D \) and \( D' \) be diagrams of \( L \) and \( L' \), and \( S \) a handle addition. Then \( S \) defines a map of cubes between the cube of resolutions of \( D \) and that of \( D' \): removing a neighbourhood of the crossings of \( D \) and \( D' \), both diagrams look identical except in the region in which the handle is attached. Moreover, as handle additions do not change the number of crossings of diagram, the resolutions of \( D \) and \( D' \) are in bijection (a string of \(0\)'s and \(1\)'s defines uniquely a resolution of \( D \) and of \( D' \)). Let the map of cubes defined by \( S \) be the map which sends a resolution of \( D \) to the associated resolution of \( D' \). As the diagrams are identical exept in a small region this map acts simply on resolutions, and depends on the handle addition contained in \( S \), and are defined as followed.

\begin{definition}
	\noindent (\emph{\(0\)-handles}): If \( S \) is a \(0\)-handle then \( \phi_S = \iota \), where \( \iota : \mathbb{Q} \rightarrow \mathcal{A} \), \( \iota ( 1 ) = \vulp \), so that \( \iota ( 1 ) \otimes \vup = v^{\text{u}}_{++} \), for example. The newly created circle cannot possess a dot, as it is contractible. Thus \( \iota  \) is \(c\)-filtered of degree \( +\frac{1}{2} \).
	
	\noindent  (\emph{\(1\)-handles}): If \( S \) is a \(1\)-handle then \( \phi_S \) acts as either \( m^0 \), \( \Delta^0 \), or \( \eta^0 \) (as defined in \Cref{Eq:dkkdiff1,Eq:dkkdiff2,Eq:dkkdiff3,Eq:dkkdiff4,Eq:dkkdiff5}) - which map is determed by the effect of \( S \) on individual resolutions. It is clear that \( \phi_S \) is \(c\)-filtered of degree \( -\frac{1}{2} \).
	
	\noindent (\emph{\(2\)-handles}): If \( S \) is a \(2\)-handle then \( \phi_S = \epsilon \), where \( \epsilon : \mathcal{A} \rightarrow \mathbb{Q} \), \( \epsilon ( \vulp ) = 0 \), \( \epsilon ( \vulm ) = 1 \). As the circle being killed is contractible, \( \epsilon \) is also \(c\)-filtered of degree \( + \frac{1}{2} \).\CloseDef
\end{definition}

As mentioned above, we can repeat this method to assign maps \( \phi^{\prime}_S : \dkh' ( L, c ) \rightarrow \dkh' ( L', c ) \) and \( \phi^{\prime \prime}_S : \dkh'' ( L, c ) \rightarrow \dkh'' ( L', c ) \) by taking filtrations and adding terms to the differential. They are defined as follows.

\begin{definition}
	\noindent (\emph{\(0\)-handles}):  \( \phi_S \), \( \phi^{\prime}_S \), and \( \phi^{\prime \prime}_S \) are all of the same form, as \( \iota \) is unchanged.
	
	\noindent  (\emph{\(1\)-handles}): \( \phi^{\prime}_S \) is obtained from \( \phi_S \) by adding the terms in parentheses in \Cref{Eq:dkkdiff1,Eq:dkkdiff2,Eq:dkkdiff3,Eq:dkkdiff4,Eq:dkkdiff5}, and \( \phi^{\prime \prime}_S \) from \( \phi^{\prime}_S \) by adding the \( d^{+2} \) terms.
	
	\noindent (\emph{\(2\)-handles}): \( \phi_S \), \( \phi^{\prime}_S \), and \( \phi^{\prime \prime}_S \) are all of the same form, as \( \epsilon \) is unchanged.\CloseDef
\end{definition}

The maps \( \phi^{\prime}_S \) and \( \phi^{\prime \prime}_S \) are of the same \(c\)-degree as \( \phi_S \). In the case of \(0\)- and \(2\)-handles this is obvious. In the case of \(1\)-handles, one can see this by noting that, although the components of \( d^{+2} \) raise the \(c\)-grading by \( \frac{3}{2} \) (as cobordism maps), we have taken an (upward) filtration, so the filtration degree of \( \phi_S \), \( \phi^{\prime}_S \), and \( \phi^{\prime \prime}_S \) depends only on terms whose \(c\)-grading is lowered.

As in the case of cobordism maps on classical or doubled Khovanov homology, \( \phi_S \) is homologically graded of degree \(0\), and quantum filtered of degree \( 0 \), \(-1\), or \(+1\), depending on its type.

In summary, we have the following maps assigned to elementary cobordisms.

\begin{definition}
	\label{Def:cobordismmaps}
	Assigned to an elementary cobordism \( S \), we have the three maps \( \phi_S : \dkh ( L, c ) \rightarrow \dkh ( L', c ) \), \( \phi^{\prime}_S : \dkh' ( L, c ) \rightarrow \dkh' ( L', c ) \) and \( \phi^{\prime \prime}_S : \dkh'' ( L, c ) \rightarrow \dkh'' ( L', c ) \): they are all trigraded of degree \( ( 0, x, \frac{1}{2} x) \), \( x \in \lbrace 0, \pm 1 \rbrace \), where \( (i,j,c) \) denotes the trigrading given by the homological, quantum, and \(c\)-gradings (in that order).\CloseDef
\end{definition}

Using these maps we can define the map assigned to a generic cobordism.

\begin{definition}
	Let \( S \hookrightarrow \Sigma_g \times I \times I \) be a cobordism between links \( L, L' \hookrightarrow \Sigma_g \times I \), such that
	\begin{equation*}
	S = S_1 \cup S_2 \cup \cdots \cup S_n
	\end{equation*}
	where \( S_i \) is an elementary cobordism. Define \( \phi_S : \dkh ( L, c) \rightarrow \dkh ( L' , c ) \) as
	\begin{equation*}
	\phi_S = \phi_{S_n} \circ \phi_{S_{n-1}} \circ \cdots \circ \phi_{S_1}
	\end{equation*}
	and likewise \( \phi^{\prime}_S : \dkh' ( L, c ) \rightarrow \dkh' ( L', c ) \) and \( \phi^{\prime \prime}_S : \dkh'' ( L, c ) \rightarrow \dkh'' ( L', c ) \).\CloseDef
\end{definition}

\begin{proposition}\label{Prop:totallyreducedconcgrading}
	Let \( S \hookrightarrow \Sigma \times I \times I \) be a concordance between a knot \( K \hookrightarrow \Sigma_g \times I \) and a link \( L \hookrightarrow \Sigma_g \times I \). Then \( \phi^{\prime \prime}_S : \dkh'' ( K, c ) \rightarrow \dkh'' ( L, c ) \) is \(c\)-filtered of degree \(0\), and has trivial kernel. If \(S\) is between two knots then \( \phi_S \) is an isomorphism.
\end{proposition}

\begin{proof}
	It is shown in \cite[Theorem \(3.19\)]{Rushworth2017} that the map on doubled Khovanov homology assigned to \( S \) is non-zero. Combining this with \Cref{Cor:doubleprimelee}, in particular the fact that the terms of the differential of \( \dkh \) are equal to those of \( \dkh '' \), we see that \( \phi^{\prime \prime}_S \neq 0 \). It is not stated explicitly in the proof of \cite[Theorem \(3.19\)]{Rushworth2017}, but this implies that \( \phi_S \) has trivial kernel: carefully applying the proof applied to cobordisms which begin with a knot (as \(S\) does here) while keeping the relationship between the elements of the canonical basis in mind makes this clear. In the case in which \( S \) is between two knots, we see that \( \phi_S \) is an injective linear map with domain and codomain a vector space of rank \(4\); by the Rank-Nullity Theorem it is surjective.
	
	To see that \( \phi^{\prime \prime}_S \) is \(c\)-filtered of degree \(0\), we recall that in any decomposition of \(S\) into elementary cobordisms the number of \(0\)- and \(2\)-handles must equal the number of \(1\)-handles. To conclude, we notice that the degree of the map assigned to a \(0\)- or a \(2\)-handle cancels exactly with that of the map assigned to \(1\)-handles (analogously to the quantum degree situation).
\end{proof}

\subsection{Obstructions to the existence of embedded discs from \( \dkh '' \)}\label{Subsec:dotteddoubledinvariants}
Let \( K \hookrightarrow \Sigma_g \times I \) be a knot in a thickened surface. All three gradings of \( \dkh '' (K,c) \) obstruct the existence of a disc bounding \(K\) in \( \Sigma_g \times I \times I \); that the homological and quantum gradings do follows from the properties of the doubled Rasmussen invariant \cite[Section \(4\)]{Rushworth2017}. In this section we show that the \(c\)-grading does also.

\begin{theorem}
	Let \( K \hookrightarrow \Sigma_g \times I \) be a knot in a thickened surface. Pick \( c \in H^1 ( \Sigma_g ; \Z_2 ) \) and compute \( \dkh '' ( K, c) \). If \( \dkh '' ( K, c) \) is not equal to that of an unknot then \( K \) does not bound a disc in \( \Sigma_g \times I \times I \).
\end{theorem}

\begin{proof}
	As mentioned above, a knot \(K\) for which \( \dkh '' ( K, c) \) has non-trivial homological or quantum degree is not slice as a virtual knot, so that in particular it cannot bound a disc in \( \Sigma_g \times I \times I \). As such, we shall focus on the case in which \( \dkh '' ( K, c) \) has trivial homological and quantum gradings, but non-trivial \(c\)-grading.
	
	First we shall show that if there exists a quantum degree, \( \mathfrak{j} \), such that for all \( x \in \dkh '' ( K , c ) \) with \( \mathfrak{j} \leq j ( x ) \) then \( c ( x ) < j( x ) - \frac{1}{2} \), then \( K \) does not bound a disc in \( \Sigma_g \times I \times I \).
	
	Let such a \( \mathfrak{j} \) exist and assume towards a contradiction that there exists a disc embedded in \( \Sigma_g \times I \times I \) which bounds \( K \). Then there exists a concordance, \( S \), from a contractible loop in \( \Sigma_g \) to \(K\), formed by cutting the disc open. Denote this loop by \( U \), so that \( \phi_S : \dkh '' ( U , c ) \rightarrow \dkh '' ( K , c ) \). The totally reduced homology of \( U \) is denoted by the black generators below:
	\begin{center}
		\begin{tikzpicture}[scale=0.8]
		
		%Axes
		\draw[black, ->] (-3,-4) -- (2,-4) ;
		\draw[black, ->] (-3,-4) -- (-3,1) ;
		
		\draw (-2,-3.8) -- (-2,-4.2) node[below] {\small $-1$} ;
		\draw (-1,-3.8) -- (-1,-4.2) node[below] {\small $-\frac{1}{2}$} ;
		\draw (0,-3.8) -- (0,-4.2) node[below] {\small $0$} ;
		\draw (1,-3.8) -- (1,-4.2) node[below] {\small $\frac{1}{2}$} ;
		
		\draw (-2.8,-3) -- (-3.2,-3) node[left] {\small $-2$} ;
		\draw (-2.8,-2) -- (-3.2,-2) node[left] {\small $-1$} ;
		\draw (-2.8,-1) -- (-3.2,-1) node[left] {\small $0$} ;
		\draw (-2.8,0) -- (-3.2,0) node[left] {\small $1$} ;
		
		%obstructions
		
		\node[] ()at(0,-3) {\color{red} $ \Q $} ;
		
		\node[] ()at(1,-2) {\color{red} $ \Q $} ;
		
		\node[] ()at(-2,-1) {\color{red} $ \Q $} ;
		
		\node[] ()at(-1,0) {\color{red} $ \Q $} ;
		
		%UnknotHomology
		
		\node[] ()at(-2,-3) {$ \Q $} ;
		
		\node[] ()at(-1,-2) {$ \Q $} ;
		
		\node[] ()at(0,-1) {$ \Q $} ;
		
		\node[] ()at(1,0) {$ \Q $} ;
		
		\end{tikzpicture}
	\end{center}
	All of the generators are at homological degree \( 0 \), and the horizontal (vertical) axis denotes the \(c\)-grading (quantum grading). By \Cref{Prop:totallyreducedconcgrading} \( \phi_S \) is \(j\)- and \(c\)-filtered of degree \( 0 \). Thus generators of \( \dkh '' ( U , c ) \) cannot decrease in either \(j\)- or \(c\)-grading, and one may think of them as being permitted to move only up and to right under the action of \( \phi_S \) (when depicted on grids such as the one above). However, as \( \mathfrak{j} \) exists, one sees that there must be generators of \( \dkh '' (U,c) \) such that there are no available generators of \( \dkh '' ( K , c) \) above and to the right of them. Also by \Cref{Prop:totallyreducedconcgrading} we have that \( \phi_S \) is an isomorphism so that such generators cannot be sent to zero, yielding the desired contradiction.
	
	To conclude we claim that if \( K \) is such that \( \dkh '' ( K , c ) \) is not that of the unknot, then it must be equal to the homology depicted by the red generators on the grid above. This can be shown using essentially identical arguments to those used to determine analogous properties of the \( j \)-grading; see \cite[Lemma \(4.2\)]{Rushworth2017}, for example. Clearly \( \mathfrak{j} = 0 \) for the homology depicted by the red generators, so that if \( K \) does not have the totally reduced homology of an unknot then it does not bound a disc in \( \Sigma_g \times I \times I \).
\end{proof}

\section{Gradings from homotopy classes and picture-valued invariants}\label{Sec:homotopygrading}

In this section we produce picture-valued invariants of links in thickened surfaces, as well as picture-valued invariants of embedded surfaces.

\subsection{Gradings from homotopy classes}\label{Subsec:homotopygradings}
We now describe a similiar process to that of \Cref{Sec:gradingdef}, now using homotopic information (rather than cohomological) to produce an additional grading. These new gradings take values in an Abelian group formed from certain diagrams, and are therefore referred to as \emph{picture-valued}. For full details we refer the reader to \cite{Manturovb}.

\begin{definition}
	\label{Def:homotopygrading}
Let \( D \) be a diagram of a link in a thickened surface \( L \hookrightarrow \Sigma_g \times I \). Denote by \( \llbracket D \rrbracket \) the standard cube of resolutions associated to \( D \). Consider the free Abelian group generated by homotopy classes of loops in \( \Sigma_g \). Denote by \( \gH \) the quotient of this group formed by equating contractible loops with the identity.

Let \( \vckh_h ( D ) \) denote the chain complex with the same chain spaces as \( \vckh ( D ) \) but an altered differential \footnote{recall that in the case of virtual Khovanov homology we have \( \eta = 0 \)}. In order to describe this new differential we define an \( \gH \)-grading on \( \vckh_h ( D ) \).

Let \( x_1 \otimes x_2 \otimes \cdots \otimes x_n = x \in \mathcal{A}^{\otimes n} \) be a state associated to the resolution \( \bigsqcup_{i=1}^n C_i \), where \( C_i \) is a circle. Define 
\begin{equation}\label{Eq:homotopygrading}
	h( x ) \coloneqq \sum_{i=1}^{n} p (x_i) [C_i] \in \gH
\end{equation}
where \( [ C_i ] \in \gH \) denotes the class represented by the \(i\)-th circle and \( p (v_\pm) = \pm 1 \). The grading \( h : \vckh_h ( D ) \rightarrow \gH \) is refered to as the \emph{homotopic grading}.

The differential of \( \vckh_h ( D ) \) is denoted \( d_h \), with components \( m_h \) and \( \Delta_h \) (the single-cycle smoothing is associated the zero map); these are the components of the maps used in virtual Khovanov homology which preserve the \( h \)-grading, and they depend on the homotopy classes represented by the circles they map between. 

Let \( m_h : \mathcal{A} \otimes \mathcal{A} \rightarrow \mathcal{A} \) correspond to the merging of circles \( C_1 \) and \( C_2 \) to produce \( C_3 \). Then
\begin{equation}\label{Eq:mh1}
m_h = \left\{\begin{array}{@{}l@{\quad}l@{}}
m, & [C_1]=[C_2]=[C_3]=\bigcirc;\\[3pt]
m^0_h, & [C_1],[C_2]\ne\bigcirc,\ [C_3]= \bigcirc;\\[3pt]
m^1_h, & [C_2]=\bigcirc,\ [C_1]=[C_3]\ne \bigcirc;\\[3pt]
m^2_h, & [C_1]=\bigcirc,\ [C_2]=[C_3]\ne\bigcirc;\\[3pt]
0, & [C_1],[C_2],[C_3]\ne\bigcirc.
\end{array}\right.
\end{equation}
The maps $m^0_h, m^1_h, m^2_h$ are defined
\begin{equation}\label{Eq:mh2}
\begin{array}{@{}l@{\quad}l@{\quad}l@{}}
m^0_h(v_+\otimes v_+)=0,   & m^1_h(v_+\otimes v_+)=v_+, & m^2_h(v_+\otimes v_+)=v_+,\\[3pt]
m^0_h(v_+\otimes v_-)=v_-, & m^1_h(v_+\otimes v_-)=0, & m^2_h(v_+\otimes v_-)=v_-,\\[3pt]
m^0_h(v_-\otimes v_+)=v_-, & m^1_h(v_-\otimes v_+)=v_-, & m^2_h(v_-\otimes v_+)=0,\\[3pt]
m^0_h(v_-\otimes v_-)=0,   & m^1_h(v_-\otimes v_-)=0, & m^2_h(v_-\otimes v_-)=0.
\end{array}
\end{equation}

Similarly, let \( \Delta_h : \mathcal{A} \rightarrow \mathcal{A} \otimes \mathcal{A} \) correspond to the splitting of the circle \( C_3 \) into the circles \( C_1 \) and \( C_2 \). Then
\begin{equation}\label{Eq:Dh1}
\Delta_h = \left\{\begin{array}{@{}l@{\quad}l@{}}
\Delta, & [C_1]=[C_2]=[C_3]=\bigcirc;\\[3pt]
\Delta^0_h, & [C_1],[C_2]\ne\bigcirc,\ [C_3]= \bigcirc;\\[3pt]
\Delta^1_h, & [C_2]=\bigcirc,\ [C_1]=[C_3]\ne \bigcirc;\\[3pt]
\Delta^2_h, & [C_1]=\bigcirc,\ [C_2]=[C_3]\ne\bigcirc;\\[3pt]
0, & [C_1],[C_2],[C_3]\ne\bigcirc.
\end{array}\right.
\end{equation}
The maps $\Delta^0_h, \Delta^1_h, \Delta^2_h$ are defined
\begin{equation}\label{Eq:Dh2}
\begin{array}{@{}l@{\quad}l@{\quad}l}
\Delta^0_h(v_+)=v_+\otimes v_- + v_-\otimes v_+,   & \Delta^1_h(v_+)=v_+\otimes v_-, & \Delta^2_h(v_+)=v_-\otimes v_+,\\[3pt]
\Delta^0_h(v_-)=0,   & \Delta^1_h(v_-)=v_-\otimes v_-, & \Delta^2_h(v_-)=v_-\otimes v_-.
\end{array}\qquad
\end{equation}
\CloseDef
\end{definition}

\begin{theorem}[\cite{Manturovb}]
	The homology of \( \vckh_h ( D ) \) with respect to \( d_h \) is an invariant of \( L \), and is denoted \( \vkh_h ( L ) \).
\end{theorem}

\Cref{Fig:hgrad21} depicts the homology of the knot given in \Cref{Fig:21lift}, split by homological, quantum, and homotopic degrees.

\begin{figure}
	\begin{tikzpicture}[scale=0.65]
	
	%alphaplus
	\node[] ()at(-0.5,3.5) { $ + \left[ \raisebox{-10pt}{\includegraphics[scale=1]{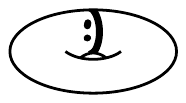}} \right] $ };
	
	\draw[black, ->] (-3.5,-4) -- (3.5,-4) ;
	\draw[black, ->] (-3.5,-4) -- (-3.5,2.5) ;
	
	\draw (-2,-3.8) -- (-2,-4.2) node[below] {\small $-2$} ;
	\draw (0,-3.8) -- (0,-4.2) node[below] {\small $-1$} ;
	\draw (2,-3.8) -- (2,-4.2) node[below] {\small $0$} ;
	
	\draw (-3.3,-3) -- (-3.7,-3) node[left] {\small $-6$} ;
	\draw (-3.3,-1.5) -- (-3.7,-1.5) node[left] {\small $-4$} ;
	\draw (-3.3,0) -- (-3.7,0) node[left] {\small $-2$} ;
	\draw (-3.3,1.5) -- (-3.7,1.5) node[left] {\small $0$} ;
	
	\node[] ()at(-2,-1.5) {$ \Z $} ;
	
	\node[] ()at(0,0) {$ \Z $} ;
	
	%alphaminus
	\node[] ()at(9.5,3.5) { $ - \left[ \raisebox{-10pt}{\includegraphics[scale=1]{alpha.pdf}} \right] $ };
	
	\draw[black, ->] (6.5,-4) -- (13.5,-4) ;
	\draw[black, ->] (6.5,-4) -- (6.5,2.5) ;
	
	\draw (8,-3.8) -- (8,-4.2) node[below] {\small $-2$} ;
	\draw (10,-3.8) -- (10,-4.2) node[below] {\small $-1$} ;
	\draw (12,-3.8) -- (12,-4.2) node[below] {\small $0$} ;
	
	\draw (6.7,-3) -- (6.3,-3) node[left] {\small $-6$} ;
	\draw (6.7,-1.5) -- (6.3,-1.5) node[left] {\small $-4$} ;
	\draw (6.7,0) -- (6.3,0) node[left] {\small $-2$} ;
	\draw (6.7,1.5) -- (6.3,1.5) node[left] {\small $0$} ;
	
	\node[] ()at(8,-3) {$ \Z $} ;
	
	\node[] ()at(10,-1.5) {$ \Z $} ;
	
	%betaplus
	\node[] ()at(-0.5,-6.5) { $ + \left[ \raisebox{-10pt}{\includegraphics[scale=1]{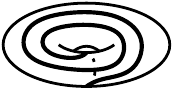}} \right] $ };
	
	\draw[black, ->] (-3.5,-14) -- (3.5,-14) ;
	\draw[black, ->] (-3.5,-14) -- (-3.5,-7.5) ;
	
	\draw (-2,-13.8) -- (-2,-14.2) node[below] {\small $-2$} ;
	\draw (0,-13.8) -- (0,-14.2) node[below] {\small $-1$} ;
	\draw (2,-13.8) -- (2,-14.2) node[below] {\small $0$} ;
	
	\draw (-3.3,-13) -- (-3.7,-13) node[left] {\small $-6$} ;
	\draw (-3.3,-11.5) -- (-3.7,-11.5) node[left] {\small $-4$} ;
	\draw (-3.3,-10) -- (-3.7,-10) node[left] {\small $-2$} ;
	\draw (-3.3,-8.5) -- (-3.7,-8.5) node[left] {\small $0$} ;
	
	\node[] ()at(2,-8.5) {$ \Z $} ;
	
	%betaminus
	\node[] ()at(9.5,-6.5) { $ - \left[ \raisebox{-10pt}{\includegraphics[scale=1]{beta.pdf}} \right] $ };
	
	\draw[black, ->] (6.5,-14) -- (13.5,-14) ;
	\draw[black, ->] (6.5,-14) -- (6.5,-7.5) ;
	
	\draw (8,-13.8) -- (8,-14.2) node[below] {\small $-2$} ;
	\draw (10,-13.8) -- (10,-14.2) node[below] {\small $-1$} ;
	\draw (12,-13.8) -- (12,-14.2) node[below] {\small $0$} ;
	
	\draw (6.7,-13) -- (6.3,-13) node[left] {\small $-6$} ;
	\draw (6.7,-11.5) -- (6.3,-11.5) node[left] {\small $-4$} ;
	\draw (6.7,-10) -- (6.3,-10) node[left] {\small $-2$} ;
	\draw (6.7,-8.5) -- (6.3,-8.5) node[left] {\small $0$} ;
	
	\node[] ()at(12,-10) {$ \Z $} ;
	
	\end{tikzpicture}
	\caption{Let \( K \) denote the knot depicted on the right of \Cref{Fig:21lift}. Above is \( \vkh_h ( K) \), split by homological, quantum, and homotopic grading. Each grid depicts the elements which are homogeneous in the homotopic grading, and their degree is given above the grid.}
	\label{Fig:hgrad21}
\end{figure}

\subsection{Maps on \( \vkh_h \)}\label{Subsec:vkhmaps}
We can assign maps on \( \vkh_h \) to cobordisms also, in essentially the same manner as in \Cref{Subsec:dkhprimemaps}; the maps assigned to \(1\)-handles are those given in \Cref{Eq:mh1,Eq:mh2,Eq:Dh1,Eq:Dh2}, and those assigned to \(0\)- and \(2\)-handles are defined, respectively, as
\begin{equation}\label{Eq:vkh0}
	\begin{aligned}
		\iota :~ &\mathbb{Q} \rightarrow \mathcal{A}\\
		 1 &\mapsto v_+
	\end{aligned}
\end{equation}
and
\begin{equation}\label{Eq:vkh2}
\begin{aligned}
\epsilon :~ &\mathcal{A} \rightarrow \mathbb{Q}\\
v_+ &\mapsto 0 \\
v_- &\mapsto 1
\end{aligned}
\end{equation}

The maps assigned to elementary cobordisms are \(h\)-graded of degree 0, by construction. They are quantum graded of degree \(+1\) and \(-1\), for \(0\)-/\(2\)-handles and \(1\)-handles, respectively.

With this in mind, given a cobordism \( S \) between links \( L_1 \) and \( L_2 \), and decomposition of it into elementary cobordisms, we define the map \( \phi_S : \vkh_h ( L_1 ) \rightarrow \vkh_h ( L_2 ) \) associated to \( S \) to be the composition of the maps assigned to the elementary cobordisms which make up the composition.

\begin{figure}
	\includegraphics[scale=0.4]{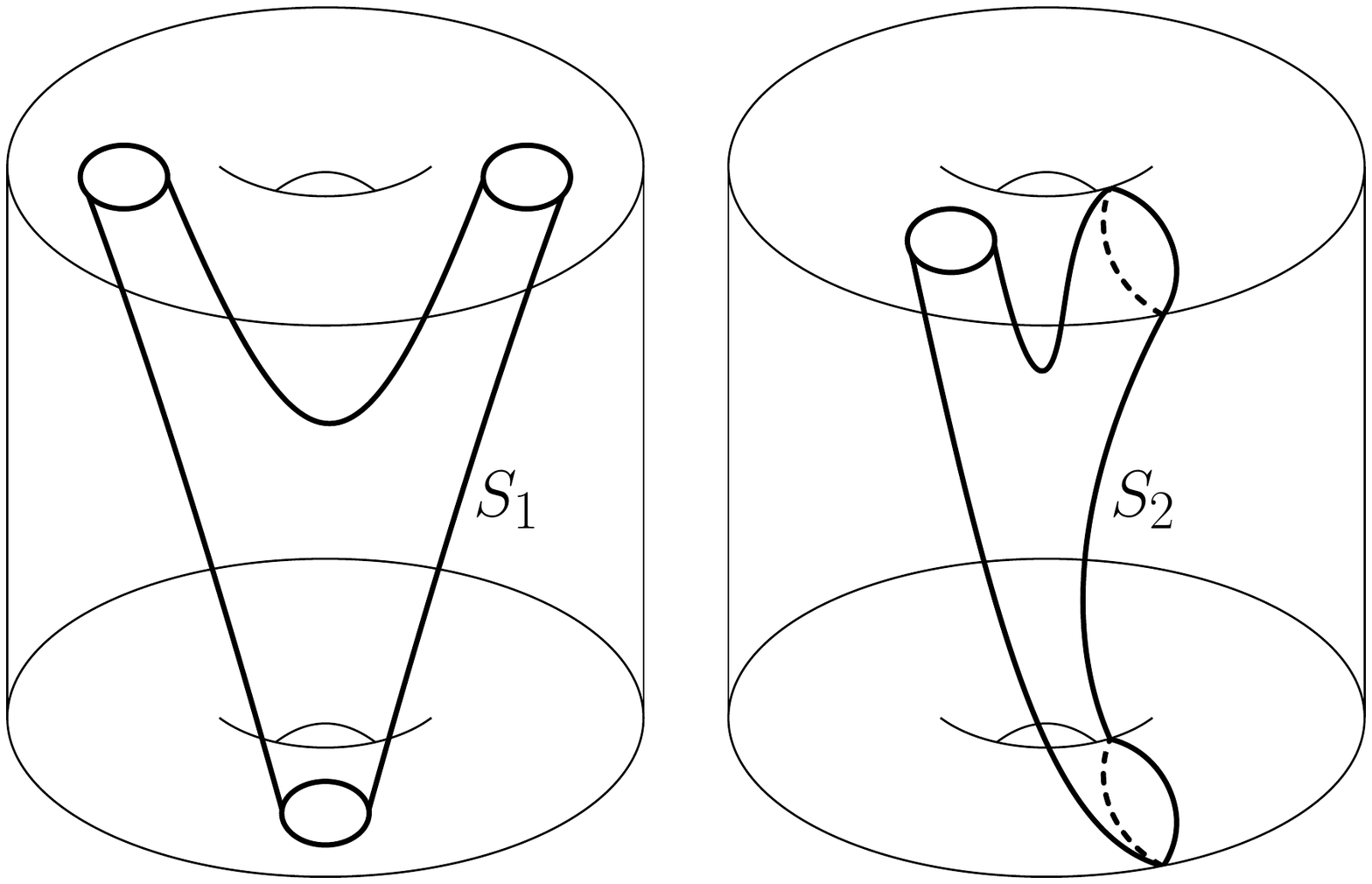}
	\caption{Two cobordisms in \( \Sigma_1 \times I \times I \).}
	\label{Fig:twopants}
\end{figure}

We conclude this section by giving an example of a cobordism which is assigned a non-trivial map on \( \vkh_h \). Consider the cobordisms depicted in \Cref{Fig:twopants}. They are both \(1\)-handle additions, and the map associated to \( S_1 \) (on the left of the figure) is \( m \), while that associated to \( S_2 \) is either \( m^1_h \) or \( m^2_h \) (depending on the ordering of the circles), as defined in \Cref{Eq:mh1}.

Further, the maps associated to \( S_1 \) and \( S_2 \) decompose according to the \(h\)-gradings of their domain and codomain. Consult the grids in \Cref{Fig:pantsgrids}: they depict the domains and codomains of \( \phi_{S_1} \) and \( \phi_{S_2} \), with quantum grading on the vertical axes, and \( h \)-grading on the horizontal axes (the homological degree of every generator is \(0\)); for convenience we denote \( \left[ \raisebox{-3pt}{\includegraphics[scale=0.4]{alpha.pdf}} \right] \) by \( \alpha \).

\begin{figure}
	\begin{tikzpicture}[scale=0.75]
	%topleft
	\draw[black, ->] (-10,-4) -- (-5.5,-4) ;
	\draw[black, ->] (-9.5,-4.5) -- (-9.5,1.5) ;
	
	\draw (-7.5,-3.75) -- (-7.5,-4.25) node[below] {\small \( 0 \)} ;
	
	\draw (-9.25,-3) -- (-9.75,-3) node[left] {\small $-2$} ;
	\draw (-9.25,-2) -- (-9.75,-2) node[left] {\small $-1$} ;
	\draw (-9.25,-1) -- (-9.75,-1) node[left] {\small $0$} ;
	\draw (-9.25,0) -- (-9.75,0) node[left] {\small $1$} ;
	\draw (-9.25,1) -- (-9.75,1) node[left] {\small $2$} ;
	
	\node[] ()at(-7.5,1) {$ \Z $} ;
	
	\node[] ()at(-7.5,-1) {$ \Z \oplus \Z $} ;
	
	\node[] ()at(-7.5,-3) {$ \Z $} ;

	%topright
	\draw[black, ->] (-2,-4) -- (2.5,-4) ;
	\draw[black, ->] (-1.5,-4.5) -- (-1.5,1.5) ;
	
	\draw (0.5,-3.75) -- (0.5,-4.25) node[below] {\small \( 0 \)} ;
	
	\draw (-1.25,-3) -- (-1.75,-3) node[left] {\small $-2$} ;
	\draw (-1.25,-2) -- (-1.75,-2) node[left] {\small $-1$} ;
	\draw (-1.25,-1) -- (-1.75,-1) node[left] {\small $0$} ;
	\draw (-1.25,0) -- (-1.75,0) node[left] {\small $1$} ;
	\draw (-1.25,1) -- (-1.75,1) node[left] {\small $2$} ;
	
	\node[] ()at(0.5,-2) {$ \Z $} ;
	\node[] ()at(0.5,0) {$ \Z $} ;
	
	%bottom right
	\draw[black, ->] (-2,-10.5) -- (2.5,-10.5) ;
	\draw[black, ->] (-1.5,-11) -- (-1.5,-5) ;
	
	\draw (-0.5,-10.25) -- (-0.5,-10.75) node[below] {\( - \alpha \)} ;
	\draw (1.5,-10.25) -- (1.5,-10.75) node[below] {\( + \alpha \)} ;
	
	\draw (-1.25,-9.5) -- (-1.75,-9.5) node[left] {\small $-2$} ;
	\draw (-1.25,-8.5) -- (-1.75,-8.5) node[left] {\small $-1$} ;
	\draw (-1.25,-7.5) -- (-1.75,-7.5) node[left] {\small $0$} ;
	\draw (-1.25,-6.5) -- (-1.75,-6.5) node[left] {\small $1$} ;
	\draw (-1.25,-5.5) -- (-1.75,-5.5) node[left] {\small $2$} ;
	
	\node[] ()at(-0.5,-8.5) {$ \Z $} ;
	\node[] ()at(1.5,-6.5) {$ \Z $} ;
	
	%bottom left
	\draw[black, ->] (-10,-10.5) -- (-5.5,-10.5) ;
	\draw[black, ->] (-9.5,-11) -- (-9.5,-5) ;
	
	\draw (-8.5,-10.25) -- (-8.5,-10.75) node[below] {\( - \alpha \)} ;
	\draw (-6.5,-10.25) -- (-6.5,-10.75) node[below] {\( + \alpha \)} ;
	
	\draw (-9.25,-9.5) -- (-9.75,-9.5) node[left] {\small $-2$} ;
	\draw (-9.25,-8.5) -- (-9.75,-8.5) node[left] {\small $-1$} ;
	\draw (-9.25,-7.5) -- (-9.75,-7.5) node[left] {\small $0$} ;
	\draw (-9.25,-6.5) -- (-9.75,-6.5) node[left] {\small $1$} ;
	\draw (-9.25,-5.5) -- (-9.75,-5.5) node[left] {\small $2$} ;
	
	\node[] ()at(-8.5,-9.5) {$ \Z $} ;
	\node[] ()at(-8.5,-7.5) {$ \Z $} ;
	\node[] ()at(-6.5,-7.5) {$ \Z $} ;
	\node[] ()at(-6.5,-5.5) {$ \Z $} ;
	
	%maps
	\draw[black, ->] (-5,-1) -- (-3,-1) node[above, pos=0.5] {\( \phi_{S_1} \)};
	\draw[black, ->] (-5,-7.5) -- (-3,-7.5) node[above, pos=0.5] {\( \phi_{S_2} \)};
	\end{tikzpicture}
	\caption{The maps assigned to the two cobordisms depicted in \Cref{Fig:twopants}, split by the homotopic grading (horizontal axis) and quantum grading (vertical axis); all generators are at homological degree \(0\).}
	\label{Fig:pantsgrids}
\end{figure}
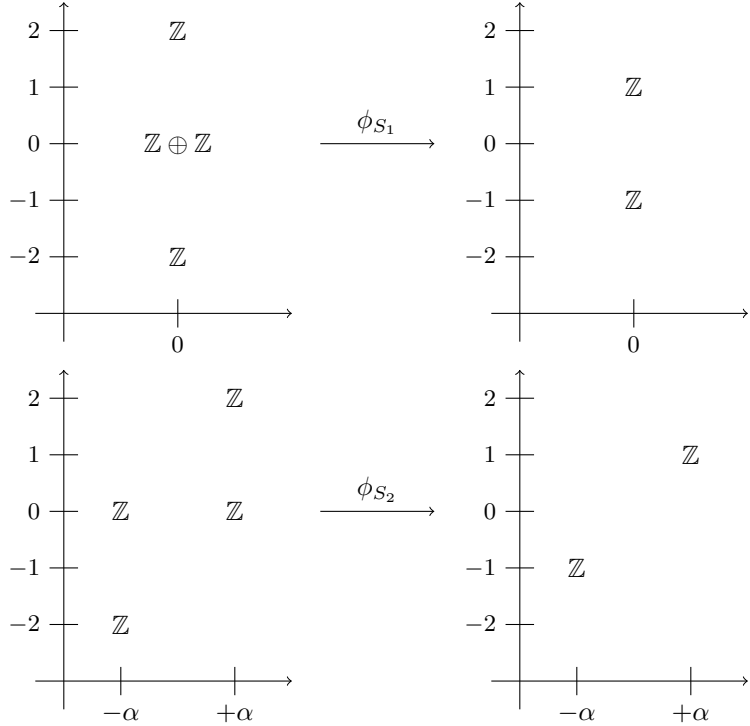

We see that \( \phi_{S_2} \) may be decomposed as \( \phi_{S_2} = \phi^{+\alpha}_{S_2} + \phi^{-\alpha}_{S_2} \), where \( \phi^{i}_{S_2} \) is supported in \(h\)-degree \(i\). In this case, \( \phi^{+\alpha}_{S_2} \) and \( \phi^{-\alpha}_{S_2} \) are both non-zero. It is clear that \( \phi_{S_1} \) does not permit such a decomposition into two non-zero maps; in \Cref{Subsec:embeddedsurfaces} we use this to show that \( S_2 \) and \( S_1 \) are not isomorphic (relative their boundaries).

\subsection{Picture-valued invariants of surfaces}\label{Subsec:embeddedsurfaces}
In this section we use the homotopic grading defined in \Cref{Subsec:homotopygradings} to produce picture-valued invariants of embedded surfaces. Specifically, to a cobordism between links in a thickened surface we can associate a map between their homologies; this map is invariant under isotopy of the cobordism relative to the boundary (in direct analogy to the case of classical Khovanov homology).

\begin{theorem}\label{Thm:picturevalued}
	Let \( S \hookrightarrow \Sigma_g \times I \times I \) be a cobordism between links \( L, L' \hookrightarrow \Sigma_g \times I \). Then the map \( \phi_S : \vkh_h ( L ) \rightarrow \vkh_h ( L' ) \) is an invariant of \( S \) up to isotopy relative the boundary.
\end{theorem}

\begin{proof}
	The proof follows that of Jacobsson in the classical case \cite{Jacobsson2004}. Specifically, one checks that distinct Morse critical points within a cobordism can be reordered, and that the construction of the map \( \phi_S \) satisfies the movie moves.
\end{proof}

Let us give some context for these picture-valued invariants, and discuss their potential applications. In \cite{Manturov2010a} the  \emph{parity bracket} is constructed. The main feature of the parity bracket, denoted \( [ \cdot ] \), is the property \emph{if a diagram is complicated enough then it realizes itself}. Namely, for an irreducible odd diagram $K$, the bracket $[K]$ equals $K$ itself meaning that for each $K'$ equivalent to $K$ we have $[K']=K$. By construction, the bracket $[K']$ contains only diagrams obtained from $K'$ by smoothings. This means that $K$ itself can be obtained from $K'$ by smoothings i.e.\ $K$ is contained \emph{inside \(K\)}.

A similar construction can be considered by taking coefficients in homotopy classes. Namely, when considering the formula for the Kauffman bracket
\begin{equation*}
	\langle K \rangle = \sum_{s} K_{s} \cdot X,
\end{equation*}
one can incorporate the homotopical information into this formula as \( X \).

Such an invariant bracket has many far-reaching consequences. If \( \langle K \rangle \) has a summand $L$ with a non-trivial coefficient, where $L$ represents some non-trivial homotopy class (say, in a thickened surface of a given genus) then such a curve  will persist in any representative of $K$.

In \cite{Manturovb} and \Cref{Subsec:homotopygradings}, the homotopy information is raised to the grading level. In this case, we see that \emph{if the homology of $K$ has a non-trivial group in $L$-grading then $L$ persists in any diagram representing a knot equivalent to $K$}.

This has some deep geometrical consequences, e.g., for a surface of a given genus, because for some free homotopy classes we can make estimates on the number of crossing points from below etc.

In \Cref{Thm:picturevalued}, we raise this picture-valued information to a higher level. Namely, having two knots $K_{1}$ and $K_{2}$ in a fixed $\Sigma_{g}\times I$ and a cobordism between them in $\Sigma_{g}\times I\times I$, we can track the homology groups of $K_{1}$ and of $K_{2}$ with picture-valued gradings. It may happen that these two knots $K_{1}$ and $_{2}$ share non-trivial homotopical (picture-valued) grading $L$. Moreover, it may happen that in a cobordism $S$ between $K_{1}$  and $K_{2}$ the map in this homotopical grading is non-trivial. As a consequence of \Cref{Thm:picturevalued}, this will mean the existence of a ``cylinder'' with homotopy class $L$.

Now, looking at $S$ as a $2$-knot with boundary, we can make a consequence that every $2$-knot with such boundary will possess similar properties; that is, the existence of cylinders.

\section{Further work}\label{Sec:furtherwork}
There are a number of directions for further work based on the material of this paper. Here we list a few of them:

\begin{enumerate}[(i)]
	\item Produce a structure theorem for the homology theories \( \mathcal{H} ( L , c ) \) and \( \dkh' ( L , c ) \) defined in \Cref{Subsec:spectralsequences} (the homology theories related to \( \dkh ( L , c ) \) by a spectral sequence in the \(c\)-grading and \(j\)-grading, respectively). These theories appear as intermediate stages between \( \dkh ( L , c ) \) and \( \dkh'' ( L , c ) \), but do not behave like either of those theories. In particular, can their structure be expressed combinatorially as in the case of doubled Lee homology, the rank of which is determined by a combinatorialy property of the argument link? Understanding the structure of \( \mathcal{H} ( L , c ) \) and \( \dkh' ( L , c ) \) could yield new concordance invariants, or could further our understanding of \( \dkh'' ( L , c ) \), for example.
	\item Use the cohomological decoration of \Cref{Sec:gradingdef} to produce a new homology theory of links in thickened surfaces: rather than doubling the module assigned to every circle, as in the standard definition of doubled Khovanov homology, is it possible to double the modules assigned to dotted circles, and leave the others unchanged? It is conceivable that such a theory would be more well-suited to links in thickened surfaces, as it would naturally contain information regarding how the link is knotted about the surface, in addition to how it is knotted about itself.
	\item Use the obstructions to the existence of discs bounding knots given in \Cref{Subsec:dotteddoubledinvariants} to produce new slice obstructions for virtual knots. There are challenges to doing so: the invariants of defined in \Cref{Sec:gradingdef} do not extend to virtual links, as the virtual Reidemeister moves can change the genus of the surface into which a representative is embedded.
	\item Related to the picture-valued invariants of \Cref{Sec:homotopygrading}: construct a Khovanov homology theory of free knots. There is a polynomial invariant of free knots \cite{Manturov2010}, can it be categorified?
\end{enumerate}

\bibliographystyle{plain}
\bibliography{library}

\end{document}